\documentclass[a4paper, 12pt]{amsart}
\usepackage{amsmath,amsthm,amssymb}
\usepackage[T1]{fontenc}
\usepackage{url}
\usepackage{amsmath,amsthm,amssymb}
\usepackage{verbatim}
\usepackage{mathrsfs}
\usepackage{graphics}
\usepackage{graphicx}
\usepackage{subfigure}
\usepackage{hyperref}
\usepackage{mathtools}
\usepackage{color}

\newtheorem{theorem}{Theorem}[section]
\newtheorem{proposition}[theorem]{Proposition}
\newtheorem{lemma}[theorem]{Lemma}
\newtheorem{corollary}[theorem]{Corollary}

\theoremstyle{remark}
\newtheorem{remark}[theorem]{Remark}
\newtheorem{definition}{Definition}

\numberwithin{equation}{section}

\newcommand{\vep}{\varepsilon}
\newcommand{\R}{{\mathbb{R}}}

\newcommand{\C}{{\mathbb{C}}}
\newcommand{\Z}{{\mathbb{Z}}}

\newcommand{\T}{{\mathbb{T}}}

\newcommand{\into}{\operatorname{int}}

\newcommand{\supp}{\operatorname{supp}}

\newcommand{\rp}{\operatorname{Re}}
\newcommand{\ip}{\operatorname{Im}}

\begin{document}

\title[Recurrence for smooth curves]
{Recurrence for smooth curves in the moduli space and application to the billiard flow on nibbled ellipses}
\author[K.\ Fr\k{a}czek]{Krzysztof Fr\k{a}czek}
\address{Faculty of Mathematics and Computer Science, Nicolaus
Copernicus University, ul. Chopina 12/18, 87-100 Toru\'n, Poland}
\email{fraczek@mat.umk.pl}




\thanks{Research partially supported by the Narodowe Centrum Nauki Grant
2017/27/B/ST1/00078}
\maketitle
\begin{abstract}
In view of classical results of Masur and Veech almost every element in the moduli space of compact translation surfaces
is recurrent, { i.e.\ its Teichm\"uller positive semiorbit returns to a compact subset infinitely many times.}
In this paper we focus on the problem of recurrence for elements of smooth curves in the moduli space.
We give an effective criterion for the recurrence of almost every element of a smooth curve. The criterion relies on
results developed by Minsky-Weiss in \cite{Mi-We}. Next we apply the criterion to the billiard flow on planar tables
confined by arcs of confocal conics. The phase space of such billiard flow splits into invariant subsets determined by caustics.
We prove that for almost every caustic the billiard flow restricted to the corresponding invariant set is uniquely ergodic.
This answers affirmatively to a question raised by Zorich.
\end{abstract}

\section{Billiards on elliptical-hyperbolic nibbled tables}
We consider a class of pseudo-integrable  billiards with piecewise elliptic and hyperbolic  boundary
introduced by Dragovi\'c and Radnovi\'c in \cite{Dra-Ra}.
Let $0<b<a$ and  denote by $\{\mathcal{C}_{\lambda}: \lambda\leq a\}$ the
family of confocal conics
\[\frac{x^2}{a-\lambda}+\frac{y^2}{b-\lambda}=1.\]
If $ \lambda< b$  then $\mathcal{C}_{\lambda}$ is an ellipse and
if $ b<\lambda< a$  then $\mathcal{C}_{\lambda}$ is a hyperbola.
Moreover, $\mathcal{C}_{b}$ is the horizontal and $\mathcal{C}_{a}$ is the vertical line though the origin.

Denote by $\Theta$ the set of sequences $(\overline{\alpha},\overline{\beta})=((\alpha_i)_{i=1}^{k},(\beta_i)_{i=1}^k)$ such that
\[a=\alpha_{0}> \alpha_{1}>\alpha_{2}>\ldots>\alpha_{k-1}>\alpha_k=b>\beta_k>\beta_{k-1}>\ldots>\beta_2>\beta_1\geq \beta_0=0. \]
Let $k(\overline{\alpha},\overline{\beta}):=k$. For every $(\overline{\alpha},\overline{\beta})\in\Theta$ let $\mathcal{D}_{\overline{\alpha},\overline{\beta}}$ be the billiard
table in the ellipse $\mathcal{C}_{0}$ so that the boundary of $\mathcal{D}_{\overline{\alpha},\overline{\beta}}$ contained in the positive quadrant is piecewise smooth and consists of a chain of arcs of ellipses $\mathcal{C}_{\beta_1},\ldots, \mathcal{C}_{\beta_k}$, hyperbolae  $\mathcal{C}_{\alpha_1},\ldots, \mathcal{C}_{\alpha_{k-1}}$ and lines $\mathcal{C}_{a},\mathcal{C}_{b}$. More precisely, the consecutive corners are intersections of the following pairs of conics:
\[(\mathcal{C}_{a},\mathcal{C}_{\beta_1}),(\mathcal{C}_{\alpha_1},\mathcal{C}_{\beta_1}),(\mathcal{C}_{\alpha_1},\mathcal{C}_{\beta_2}),
(\mathcal{C}_{\alpha_2},\mathcal{C}_{\beta_2}),\ldots
, (\mathcal{C}_{\alpha_{k-1}},\mathcal{C}_{\beta_{k-1}}),(\mathcal{C}_{\alpha_{k-1}},\mathcal{C}_{\beta_{k}}), (\mathcal{C}_{b},\mathcal{C}_{\beta_k}).\]
The positive quadrant of $\mathcal{D}_{\overline{\alpha},\overline{\beta}}$ looks like stairs whose steps are elliptical-hyperbolic, see Figure~\ref{fig:Dalphabeta}.
\begin{figure}[h]
\includegraphics[width=0.4\textwidth]{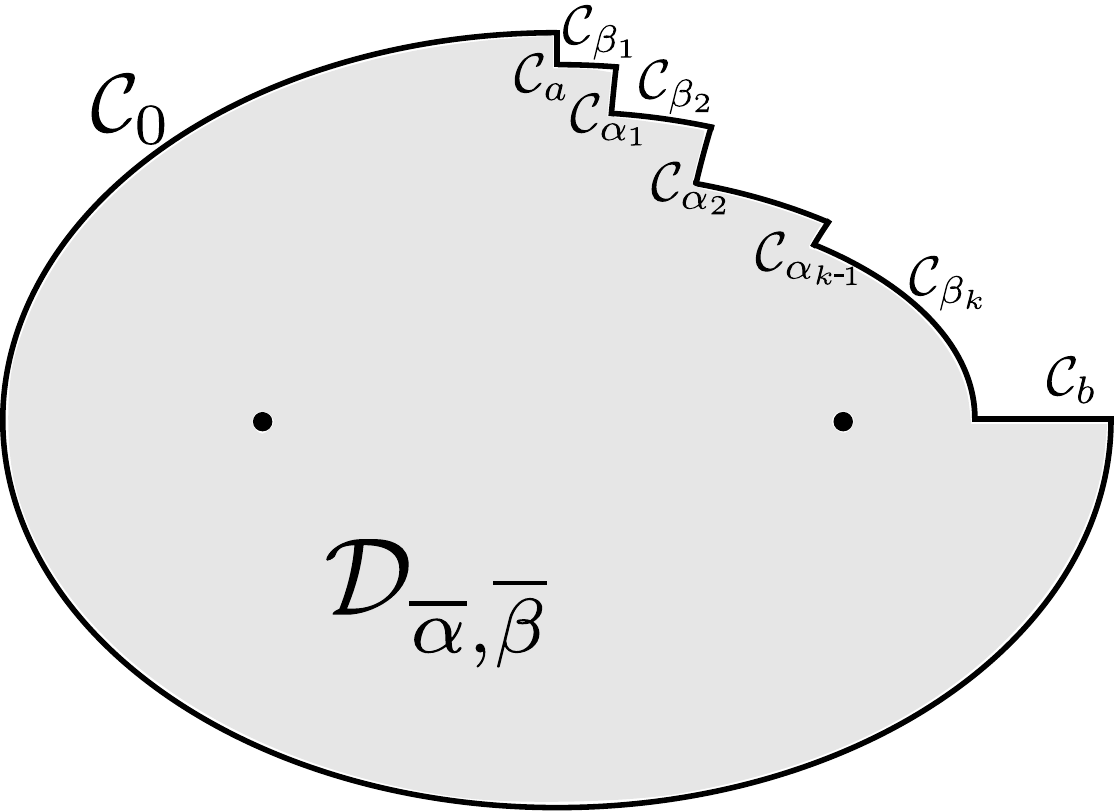}
\caption{The shape of the table $\mathcal{D}_{\overline{\alpha},\overline{\beta}}$.}\label{fig:Dalphabeta}
\end{figure}

Let $(\overline{\alpha}^{++},\overline{\beta}^{++})$, $(\overline{\alpha}^{+-},\overline{\beta}^{+-})$, $(\overline{\alpha}^{-+},\overline{\beta}^{-+})$, $(\overline{\alpha}^{--},\overline{\beta}^{--})$ be sequences in
$\Theta$ such that
\[\beta^{++}_1=\beta^{-+}_1,\ \beta^{+-}_1=\beta^{--}_1,\ \beta^{++}_{k(\overline{\alpha}^{++},
\overline{\beta}^{++})}=\beta^{+-}_{k(\overline{\alpha}^{+-},\overline{\beta}^{+-})},\
\beta^{-+}_{k(\overline{\alpha}^{-+},
\overline{\beta}^{-+})}=\beta^{--}_{k(\overline{\alpha}^{--},\overline{\beta}^{--})}.\]
Let
\[\beta^t:=\beta^{++}_1=\beta^{-+}_1,\quad\beta^b:=\beta^{+-}_1=\beta^{--}_1,\]
\[\beta^r:=\beta^{++}_{k(\overline{\alpha}^{++},
\overline{\beta}^{++})}=\beta^{+-}_{k(\overline{\alpha}^{+-},\overline{\beta}^{+-})},\quad
\beta^l:=\beta^{-+}_{k(\overline{\alpha}^{-+},
\overline{\beta}^{-+})}=\beta^{--}_{k(\overline{\alpha}^{--},\overline{\beta}^{--})}.\]
Denote by $\gamma_v, \gamma_h:\R^2\to\R^2$ the reflections across the vertical and the horizontal coordinate axis respectively.
For the quadruple $(\overline{\alpha}^{++},\overline{\beta}^{++})$, $(\overline{\alpha}^{+-},\overline{\beta}^{+-})$, $(\overline{\alpha}^{-+},\overline{\beta}^{-+})$, $(\overline{\alpha}^{--},\overline{\beta}^{--})$ let
\begin{align}\label{tableD}
\begin{aligned}
\mathcal{D}&=\mathcal{D}^{(\overline{\alpha}^{++},\overline{\beta}^{++})(\overline{\alpha}^{+-},\overline{\beta}^{+-})}_{ (\overline{\alpha}^{-+},\overline{\beta}^{-+})(\overline{\alpha}^{--},\overline{\beta}^{--})}\\:&=
\mathcal{D}_{\overline{\alpha}^{++},\overline{\beta}^{++}}\cap
\gamma_h\mathcal{D}_{\overline{\alpha}^{+-},\overline{\beta}^{+-}}\cap
\gamma_v\mathcal{D}_{\overline{\alpha}^{-+},\overline{\beta}^{-+}}
\cap\gamma_h\circ\gamma_v\mathcal{D}_{\overline{\alpha}^{--},\overline{\beta}^{--}}.
\end{aligned}
\end{align}
Then every quadrant of $\mathcal{D}$ looks like stairs whose steps are elliptical-hyperbolic, see Figure~\ref{fig:Dgen}.
We call the table $\mathcal D$ a \emph{nibbled ellipse}.
\begin{figure}[h]
\includegraphics[width=0.4\textwidth]{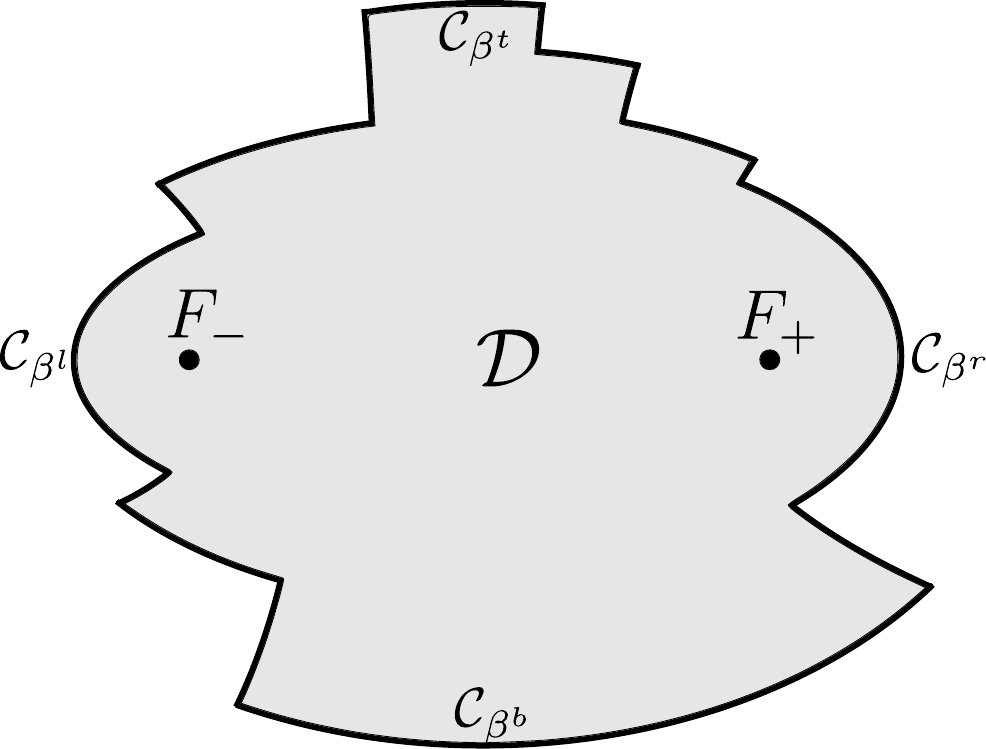}
\caption{The shape of the table $\mathcal{D}=\mathcal{D}^{(\overline{\alpha}^{++},\overline{\beta}^{++})(\overline{\alpha}^{+-},\overline{\beta}^{+-})}_{ (\overline{\alpha}^{-+},\overline{\beta}^{-+})(\overline{\alpha}^{--},\overline{\beta}^{--})}$.}\label{fig:Dgen}
\end{figure}

{ Let us consider the billiard flow $(b_t)_{t\in\R}$ on the billiard table $\mathcal D$ which acts on unit tangent vectors
$(x,\theta)\in S^1\mathcal D\subset\mathcal D\times S^1$. The flow $(b_t)_{t\in\R}$
moves $(x,\theta)$ at unit speed along the straight line through the foot point
$x\in \mathcal D$ in  direction $\theta\in S^1$ with elastic collisions at the
boundary of the table (according with the law that the angle of
incidence equals the angle of reflection with respect to the tangent at the collision point).
After reaching any of the corners, the billiard flow dies.}

Dragovi\'c and Radnovi\'c  observed in \cite{Dra-Ra} that the
phase space $S^1\mathcal D$ of the billiard flow on $\mathcal{D}$
splits into invariant subsets $\mathcal{S}_{s}$,
$s\in(\min\{\beta^t,\beta^b\},a]$ so that the ellipse $\mathcal{C}_{s}$ for
$\min\{\beta^t,\beta^b\}<s< b$  or the hyperbola $\mathcal{C}_{s}$ for
$b<s<a$  is a caustic\footnote{ Caustic is a curve for which tangent billiard trajectories remains tangent after successive reflections.} of all billiard trajectories in
$\mathcal{S}_{s}$ (see Figure~\ref{fig:Splitting}).

For every $s\in(0,b)$ denote by $\mathcal E_s$ the set of external points of the ellipse $\mathcal{C}_s$ and for every $s\in (b,a)$
denote by $\mathcal H_s$ the area between two branches of the hyperbola $\mathcal{C}_s$. Then { every billiard orbit in $\mathcal{S}_{s}$ is trapped in the set $\mathcal{D}\cap \mathcal E_s$ for  $\min\{\beta^t,\beta^b\}<s< b$ or $\mathcal{D}\cap \mathcal H_s$ for  $b<s< a$. Therefore,  the set of foot points (denoted by $S_s$) of vectors in $\mathcal S_s$} can be identified
with $\mathcal{D}\cap \mathcal E_s$ for  $\min\{\beta^t,\beta^b\}<s< b$ and  with $\mathcal{D}\cap \mathcal H_s$ for  $b<s< a$.
If $\max\{\beta^t,\beta^b\}<s<\min\{\beta^l,\beta^r\}$
then the set $\mathcal{S}_{s}$ slits into two connected  sets: the upper one $\mathcal{S}^+_{s}$ and the lower one $\mathcal{S}^-_{s}$, see Figure~\ref{fig:Splitting}.
\begin{figure}[h]
\includegraphics[width=1.0 \textwidth]{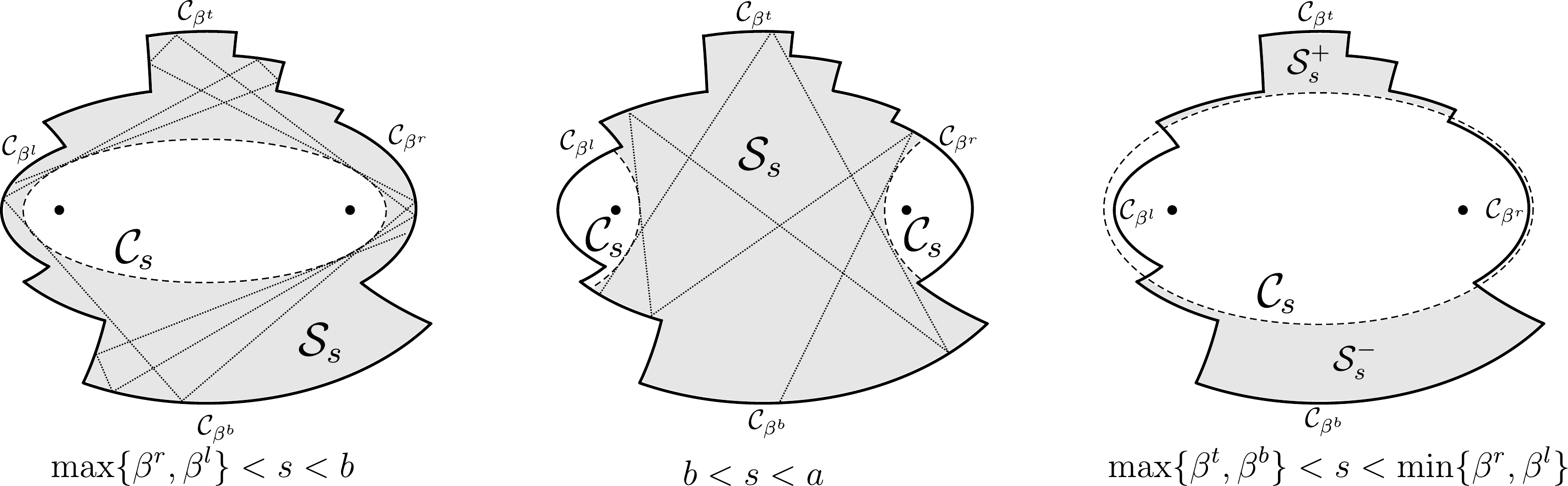}
\caption{Invariant subsets of the phase space.}\label{fig:Splitting}
\end{figure}

\smallskip

The aim of the paper is to answer affirmatively to the conjecture, raised by Zorich,
that  for almost all parameters $s$
all billiard orbits
in  $\mathcal{S}_{s}$ (or in $\mathcal{S}^\pm_{s}$) are equidistributed in $\mathcal{S}_{s}$ { (or in $\mathcal{S}^\pm_{s}$ resp.)}.

{ Recall that an abstract Borel flow $(T_t)_{t\in\R}$ on a metric space $X$ is \emph{uniquely ergodic} (or all its orbits are \emph{equidistributed} in $X$)
if there exists a probability Borel measure $\mu$ on $X$ such that for every compactly supported continuous map $f:X\to\C$ and every $x\in X$ we have
\[\frac{1}{T}\int_{0}^Tf(T_tx)\to \int_Xf\,d\mu.\]
Then $\mu$ is the unique probability invariant measure of the flow $(T_t)_{t\in\R}$.}
\begin{theorem}\label{thm:main}
For every nibbled ellipse $\mathcal D$ of the form \eqref{tableD} and for almost all  $s\in(\min\{\beta^t,\beta^b\},a)$
 the billiard flow $(b_t)_{t\in\R}$ on $\mathcal D$
restricted to any connected component of $\mathcal{S}_{s}$ is uniquely ergodic.
\end{theorem}

Recall that the same result was proved in \cite{Fr-Shi-Ul} for a special degenerate family of nibbled ellipses, i.e.\ for ellipses with a linear obstacle. The first step of the proof (in \cite{Fr-Shi-Ul} and in the present paper) is to consider  a special change of variables $\sigma_s$ (introduced in \cite{Dra-Ra}) leading to a polygonal billiard table $\sigma_s(S_s)$ with vertical and horizontal sides. After the change of variables the billiard flow $(b_t)_{t\in\R}$ on $\mathcal S_s$ becomes the directional billiard flow in directions $\pm\pi/4$, $\pm3\pi/4$ on $\sigma_s(S_s)$. Since $\sigma_s(S_s)$ is a rational
polygon, the map $s\mapsto \sigma_s(S_s)$ provides (after an unfolding procedure) a curve in the moduli space $\mathcal M$ of translation surfaces.

In  \cite{Fr-Shi-Ul} the unique ergodicity of the directional flows followed from the fact that almost every element of the corresponding  curve is Birkhoff ergodic for the Teichm\"uller flow $(g_t)_{t\in\R}$ restricted to an appropriate $SL_2(\R)$-invariant subsets of $\mathcal M$. In the present paper we apply a different approach developed in \cite{Mi-We}.

\subsection{Change of variables $\sigma_s$.}\label{sec:change}
First notice that  each point of the non-negative quadrant $\R^2_{\geq 0}$ except the focus $F_+$ is the intersection point of two conics $\mathcal{C}_{\lambda_1}$, $\lambda_1\in[b,a]$ and
$\mathcal{C}_{\lambda_2}$, $\lambda_2\in(-\infty,b]$.  This gives a { coordinate system $(\lambda_1,\lambda_2)\in [b,a]\times[-\infty,b]\setminus \{(b,b)\}$ in the set $\R^2_{\geq 0}\setminus\{F_+\}$}. In { this coordinate system} the elliptic and hyperbolic arcs forming
the boundary of the table are horizontal or vertical linear segments.

Let $e(\lambda,s):=\tfrac{1}{\sqrt{(a-\lambda)(b-\lambda)(s-\lambda)}}$.
For any $s\in(-\infty,b)$ let us consider
{ a new coordinate system} in $\mathcal E_s\cap \R^2_{\geq 0}$
determined by
\[\sigma_s(\lambda_1,\lambda_2)=\Big(\int_{\lambda_1}^ae(\lambda,s)\,d\lambda,\int_{\lambda_2}^se(\lambda,s)\,d\lambda\Big).\]
The domain of the new { coordinate system} is $[0,\ell(s)]\times[0,\ell(s))$, where
\[\ell(s):=\int_b^ae(\lambda,s)\,d\lambda=\int_{-\infty}^se(\lambda,s)\,d\lambda.\]
The { new coordinate system} extends by symmetry to the whole annulus $\mathcal E_s$,  its domain is the cylinder $\R/4\ell(s)\Z\times[0,\ell(s))$. More precisely, the extended coordinate chart $\sigma_s:\mathcal E_s\to \R/4\ell(s)\Z\times[0,\ell(s))$ is determined by
\begin{align*}
&\sigma_s|_{\mathcal E_s\cap \R_{\leq 0}\times\R_{\geq 0}}=Tr_{-(\ell(s),0)}\circ\sigma_s\circ \gamma_v, \quad
\sigma_s|_{\mathcal E_s\cap \R_{\leq 0}\times\R_{\leq 0}}=Tr_{-(2\ell(s),0)}\circ\sigma_s\circ \gamma_h\circ\gamma_v, \\
&\sigma_s|_{\mathcal E_s\cap \R_{\geq 0}\times\R_{\leq 0}}=Tr_{-(3\ell(s),0)}\circ\sigma_s\circ \gamma_h,
\end{align*}
where $Tr_v$ is the translation by the vector $v$.

One can carry out similar construction of { a coordinate system in} the sets $\mathcal H_s$, $s\in(b,a)$ starting from { the coordinate system in}
$\mathcal H_s\cap \R^2_{\geq 0}$ given by
 \[\sigma_s(\lambda_1,\lambda_2)=\Big(\int_{\lambda_1}^ae(\lambda,s)\,d\lambda,\int_{\lambda_2}^be(\lambda,s)\,d\lambda\Big).\]
Then the domain of the coordinate system in $\mathcal H_s\cap \R^2_{\geq 0}$ is $[0,\ell(s)]\times [0,\ell(s))$, where
\[\ell(s):=\int_s^ae(\lambda,s)\,d\lambda =\int_{-\infty}^b e(\lambda,s)\,d\lambda.\]
The domain of the extended { coordinate system (in} $\mathcal H_s$) is  $[-\ell(s),\ell(s)]\times (-\ell(s),\ell(s))$ and
the coordinate chart $\sigma_s:\mathcal H_s\to [-\ell(s),\ell(s)]\times (-\ell(s),\ell(s))$ is determined by
\begin{align*}
&\sigma_s|_{\mathcal H_s\cap \R_{\leq 0}\times\R_{\geq 0}}=\gamma_v\circ\sigma_s\circ \gamma_v, \quad
\sigma_s|_{\mathcal H_s\cap \R_{\leq 0}\times\R_{\leq 0}}=\gamma_h\circ\gamma_v\circ\sigma_s\circ \gamma_h\circ\gamma_v, \\
&\sigma_s|_{\mathcal H_s\cap \R_{\geq 0}\times\R_{\leq 0}}=\gamma_h\circ\sigma_s\circ \gamma_h.
\end{align*}

Recall that $ S_s=\mathcal D\cap\mathcal E_s$ for $\min\{\beta^t,\beta^b\}<s<b$ and
$S_s=\mathcal D\cap\mathcal H_s$ for $b<s<a$. Each set $S_s$ is
regarded in separate coordinates given by the coordinate chart $\sigma_s$.
Then $\sigma_s(S_s)$ is a polygon with vertical and horizontal
sides in $\R^2$ or in the cylinder $\R/4\ell(s)\Z\times\R$. Moreover, by Proposition~5.2 in  \cite{Dra-Ra}, we have the following result.

\begin{proposition}\label{prop:bilflow}
For every nibbled ellipse $\mathcal D$
the billiard flow $(b_t)_{t\in\T}$ on $S^1\mathcal D$ restricted to $\mathcal S_s$, $s\in (\min\{\beta^t,\beta^b\},b)\cup(b,a)$ is
{ conjugate (by $\sigma_s$) to} the directional billiard flow on $\sigma_s(S_s)$ in directions $\pm\pi/4$, $\pm 3\pi/4$.
\end{proposition}

A precise description of the polygons $\sigma_s(S_s)$ for $s\in (\min\{\beta^t,\beta^b\},b)\cup(b,a)$ is presented in Section~\ref{sec:uniqerg}. In Section~\ref{sec:poly} we supply an appropriate language for this description.

\smallskip
{
Formally the directional billiard flow on $\sigma_s(S_s)$ in directions $\pm\pi/4$, $\pm 3\pi/4$ acts on the union of four copies of the polygon, denoted by  $\sigma_s(S_s)_{\pi/4}$,
$\sigma_s(S_s)_{-\pi/4}$, $\sigma_s(S_s)_{3\pi/4}$, $\sigma_s(S_s)_{-3\pi/4}$. Each copy $\sigma_s(S_s)_{\theta}$ for $\theta\in \{\pm\pi/4$, $\pm 3\pi/4\}$ represents
all unit vectors pointing in the same direction $\theta$. After applying the horizontal or vertical reflection (or both) to each copy separately, we can arrange all unit vectors to point to the same direction $\pi/4$.
More precisely, after such transformations, all unit vectors in $\sigma_s(S_s)_{\pi/4}$, $\gamma_h\sigma_s(S_s)_{-\pi/4}$, $\gamma_v\sigma_s(S_s)_{3\pi/4}$ and $\gamma_h\circ\gamma_v\sigma_s(S_s)_{-3\pi/4}$
point to the same direction $\pi/4$. By gluing corresponding sides of these four polygons, we get a compact connected orientable surface $M_s$ with a translation structure $\omega(s)$ inherited from the Euclidian plan.
Moreover, the directional billiard flow on $\sigma_s(S_s)$ in directions $\pm\pi/4$, $\pm 3\pi/4$ is conjugate to the translation flow in direction $\pi/4$ (denoted by $(\varphi^{\pi/4}_t)_{t\in\R}$) on the translation surface $(M_s,\omega(s))$. This is an example of using the previously mentioned unfolding procedure coming from \cite{Fox-Ker} and \cite{Ka-Ze}.

A precise description of the polygons $\sigma_s(S_s)$ for $s\in (\min\{\beta^t,\beta^b\},b)\cup(b,a)$ presented in Section~\ref{sec:uniqerg}, shows that the interval $(\min\{\beta^t,\beta^b\},a)$
splits into finitely many subintervals $\{J:J\in\mathcal{J}\}$ so that for all $s$'s form the interior of $J\in\mathcal J$ the surfaces $M_s$ have the same genus $g_J$ and the map $J\ni s\mapsto \omega(s)$ is smooth.
In view of Proposition~\ref{prop:bilflow}, it follows that we need to prove that for every $J\in\mathcal{J}$ and for a.e.\ $s\in J$ the directional flow $(\varphi^{\pi/4}_t)_{t\in\R}$ on the translation surface $(M_{g_J}, \omega(s))$ is uniquely ergodic, where $M_{g_J}$ is a compact connected orientable surface of genus $g_J$. This observation allows us to translate the original problem into the language of translational surfaces and smooth curves in the moduli space of translational surfaces.}

\section{Translation surfaces}
\begin{definition}\label{def:transsurf}
A \emph{translation surface} is a compact connected orientable topological surface $M$, together with a finite set
of points $\Sigma$ (singularities) and an atlas of charts $\omega=\{\zeta_\alpha:U_\alpha\to \C:\alpha\in\mathcal{A}\}$ on $M\setminus \Sigma$ such that every transition map
$\zeta_\beta\circ\zeta^{-1}_\alpha:\zeta_\alpha(U_\alpha\cap U_\beta)\to \zeta_\beta(U_\alpha\cap U_\beta)$ is a translation, i.e.\ for every connected component $C$ of $U_\alpha\cap U_\beta$ there exists
$v_{\alpha,\beta}^C\in \C$ such that  $\zeta_\beta\circ\zeta^{-1}_\alpha(z)=z+v^C_{\alpha,\beta}$ for $z\in \zeta_\alpha^{-1}(C)$.
\end{definition}

{
For every $\theta\in\R/2\pi\Z$ (we will identify $\R/2\pi\Z$ with $S^1$) let $X_\theta$ be a tangent vector field on $M\setminus\Sigma$ which is the pullback of the unit constant vector field $e^{i\theta}$ on $\C$ through the charts of the atlas.
Since the derivative of any transition map is the identity, the vector field  $X_\theta$ is well defined on $M\setminus\Sigma$.
Denote by $(\varphi^\theta_t)_{t\in\R}$ the corresponding flow, called the translation flow on $(M,\omega)$ in direction $\theta$. The flow preserves the measure $\lambda_{\omega}$ which is the pullback of the Lebesgue measure on $\C$. We will denote by $(\varphi^v_t)_{t\in\R}$ and $(\varphi^h_t)_{t\in\R}$ the vertical and horizontal flow respectively.

A \emph{saddle connection} in direction $\theta$ is an orbit segment of $(\varphi^\theta_t)_{t\in\R}$ that goes
from a singularity to a singularity (possibly, the same one) and has no interior singularities. A semiorbit of $(\varphi^\theta_t)_{t\in\R}$ that goes from or to a singularity is  called  an outgoing or incoming \emph{separatrix}. Recall that if $(M,\omega)$ has no saddle connection in direction $\theta$, then the flow $(\varphi^\theta_t)_{t\in\R}$ is \emph{minimal}, i.e.\
every its orbit is dense in $M$.

\smallskip

Given a topological compact connected orientable surface $M$ and its finite subset $\Sigma\subset M$,
denote by $\operatorname{Diff}^+(M,\Sigma)$ the group of orientation-preserving homeomorphisms of $M$ which fix all elements of $\Sigma$.
Denote by $\mathcal{M}(M,\Sigma)$ the \emph{moduli space} of translation surfaces with singularities at $\Sigma$, i.e.\
the space of orbits of the natural action of
$\operatorname{Diff}^+(M,\Sigma)$ on the space of translation structures on $M$ with singularities at $\Sigma$.
The moduli space has a natural structure of complex orbifold (locally the quotient of a
complex manifold by a finite group) described in detail in \cite{Yoc}}

On the moduli space the Teichm\"uller flow $(g_t)_{t\in\R}$ acts deforming the translation structure $\omega$ in local coordinates by linear maps $\{\begin{bsmallmatrix}
 e^t & 0 \\
 0 & e^{-t}
 \end{bsmallmatrix}:t\in\R\}$
and the rotations $(r_{\theta})_{\theta\in\R/\Z}$ act by linear maps
$\{\begin{bsmallmatrix}
 \cos \theta & -\sin\theta \\
 \sin\theta & \cos\theta
 \end{bsmallmatrix}:t\in\R\}$.

{
\begin{remark}\label{rem:rot}
Notice that for every $(M,\omega)\in \mathcal{M}(M,\Sigma)$ and any $\theta\in S^1$ the directional flow $(\varphi^{\theta}_t)$ on $(M,\omega)$
coincides with the vertical flow $(\varphi^{v}_t)$ on $(M,r_{\pi/2-\theta}\omega)$.
\end{remark}

\begin{definition}
A translation surface $(M,\omega)\in \mathcal{M}(M,\Sigma)$ is called \emph{recurrent} if there exists a sequence $(t_n)_{n\geq 1}$ increasing to $+\infty$  and a compact subset $\mathcal K\subset \mathcal{M}(M,\Sigma)$
such that $g_{t_n}(M,\omega)\in\mathcal K$ for all $n\geq 1$.
\end{definition}
\begin{proposition}[Masur \cite{Ma}]\label{prop:mas}
If a translation surface $(M,\omega)$ is recurrent, then
the vertical flow on $(M,\omega)$ is uniquely ergodic.
\end{proposition}
}

One of the main aims of the paper is to formulate and prove an effective criterion for the recurrence of almost every element of a smooth curve in the moduli space $\mathcal{M}(M,\Sigma)$.
More precisely,  we deal with a $C^\infty$-map $J\ni s\mapsto \omega(s)\in \mathcal{M}(M,\Sigma)$ ($J\subset\R$ is a finite interval) and we want to show
that  $(M,\omega(s))$ is recurrent for a.e.\ $s\in J$.
{ In fact, we want to use this type of result for the $r_{\pi/4}$-rotation of curves mentioned at the end of Section~\ref{sec:change}.
Indeed, in view of Remark~\ref{rem:rot} and Proposition~\ref{prop:mas}, if a.e.\ element of the curve $s\mapsto r_{\pi/4}\omega(s)$ is recurrent, then
for a.e.\ $s$ the flow $(\varphi^{\pi/4}_t)_{t\in\R}$ on $(M,\omega(s))$ is uniquely ergodic.

An archetypical example of the criterion for recurrence is a classical theorem by  Kerckhoff, Masur and Smillie
\cite{Ke-Ma-Sm} saying that for every compact translation surface $(M,\omega)$ the rotated translation surface $(M,r_s\omega)$ is recurrent for a.e.\ $s\in[0,2\pi]$, i.e.\ here we deal with specific curves of the form $[0,2\pi]\ni s\mapsto r_s\omega\in \mathcal{M}(M,\Sigma)$.
However, this result does not apply to the $r_{\pi/4}$-rotation of curves mentioned at the end of Section~\ref{sec:change}.}

Another important step toward understanding the problem of recurrence was made by Minsky and Weiss in \cite{Mi-We02} where recurrence is shown for a.e.\ element of any horocyclic arc in $\mathcal{M}(M,\Sigma)$.
The ideas developed in \cite{Mi-We02} were further extended  in \cite{Mi-We} to curves well approximated by horocylic arcs
{ and then used to prove a criterion for a.e.\ recurrence for curves of interval exchange transformations (see Theorem~\ref{thm:Mi-Weiss}).
The main aim of this section is to reformulate and prove Minsky-Weiss criterion in terms of translation surfaces and their relative homologies (see Theorem~\ref{thm:gencrit})}.


\smallskip

{ The transition from translation surfaces to interval exchange transformations is obvious and consists in choosing a transversal section to the vertical flow and considering the map of the first return.}
Suppose that  a horizontal interval $I\subset M$ is a global transversal for the vertical flow $(\varphi^v_t)_{t\in\R}$ on $(M,\omega)$, i.e.\
its every infinite semiorbit meets $I$ infinitely many times. Recall that this condition holds for any horizontal interval whenever $(M,\omega)$ has no vertical saddle connection.
The interval $I$ we identify with the real interval $[0,|I|)$. Denote by $T_{\omega,I}:I\to I$ the first return map of the flow  $(\varphi^v_t)_{t\in\R}$ to $I$.
Then $T_{\omega,I}$ is an interval exchange transformation whose discontinuities belong to incoming separatices.

\smallskip

Every $d$-interval exchange transformation (IET) is determined by two parameters: a permutation $\pi\in S_d$ ($S_{d}$ is the group of permutations of the set $\{1,\ldots,d\}$) and $\lambda\in\R^d_{>0}$ as follows.
For every $\Lambda=(\pi,\lambda)\in S_{d}\times\R^d_{>0}$ let
\[b_j(\Lambda)=\sum_{i\leq j}\lambda_i,\quad t_j(\Lambda)=\sum_{\pi(i)\leq j}\lambda_i\quad\text{for}\quad 0\leq j\leq d.\]
Then
\[b_{0}(\Lambda)=t_{0}(\Lambda)=0,\quad b_{d}(\Lambda)=t_{d}(\Lambda)=|\lambda|=\sum_{i=1}^d\lambda_i\]
and
\[b_{j}(\Lambda)-b_{j-1}(\Lambda)=\lambda_j,\,t_{j}(\Lambda)-t_{j-1}(\Lambda)=\lambda_{\pi^{-1}(j)}\text{ for }1\leq j\leq d.\]
Denote by $T_\Lambda=T_{\pi,\lambda}:[0,|\lambda|)\to[0,|\lambda|)$ the IET so that each interval $[b_{j-1}(\Lambda),b_{j}(\Lambda))$ is translated by $T_\Lambda$ into  $[t_{\pi (j)-1}(\Lambda),t_{\pi (j)}(\Lambda))$, i.e.
\begin{equation}\label{eq:defT}
T_\Lambda x=x+t_{\pi (j)}(\Lambda)-b_{j}(\Lambda)\quad\text{for all}\quad x\in [b_{j-1}(\Lambda),b_{j}(\Lambda)).
\end{equation}

We say that the  IET $T_{\Lambda}$ has a \emph{connection} if there exist $1\leq i,j< d$ and $n>0$ such that $T^n_{\Lambda}b_i(\Lambda)=b_j(\Lambda)$. The IET $T_{\Lambda}$ is \emph{uniquely ergodic} if
it has no connection and the Lebesgue measure on $[0,|\lambda|)$ is the only $T_\Lambda$-invariant measure.
For every $n>0$ denote by $\vep_n(\Lambda)$ the minimal distance between the points $T^k_\Lambda b_i(\Lambda)$ for $0\leq k\leq n$ and $1\leq i<d$.
We say that the IET $T_\Lambda$ is of \emph{recurrence type} if it has no connection and $\liminf n\vep_n(\Lambda)>0$. Recall that every IET of recurrence type is uniquely ergodic.

\begin{remark}\label{rem:unique}
Suppose that a horizontal interval $I\subset M$ is a global transversal for the vertical flow $(\varphi^v_t)_{t\in\R}$ on $(M,\omega)$.
If the flow $(\varphi^v_t)_{t\in\R}$ has no saddle connection then $T_{\omega,I}$ has no connection.
Moreover, in view of \cite[Sect.\ 3.3]{Vor} (see \cite[Proposition 7.2]{Mi-We} for a qualitative version), the translation surface $(M,\omega)$ is recurrent if and only if
the IET $T_{\omega,I}$ is of recurrent type.
\end{remark}

\subsection{Minsky-Weiss approach and its application}

{ Let $J\ni s\mapsto (M,\omega(s))\in \mathcal{M}(M,\Sigma)$ be a $C^{\infty}$ map.}
Suppose that for every $s\in J$ there exists a horizontal interval $I_s$ in $(M,\omega(s))$ so that $I_s$ is a global transversal for  the vertical flow on $(M,\omega(s))$
and $s\mapsto I_s$ is of class $C^\infty$. Assume that all IETs $T_s:=T_{\omega(s),I_s}$ for $s\in J$ exchange $d\geq 2$ intervals according to the same
permutation $\pi\in S_d$. Then there exists a $C^\infty$-map $J\ni s\mapsto \Lambda(s)=(\pi,\lambda(s))\in S_{d}\times\R^d_{>0}$ such that $T_s=T_{\Lambda(s)}$ for every $s\in J$.  For every $s\in J$ we define a piecewise constant function $L_s:I_s\to\R$ by
\begin{equation}\label{eq:defL}
L_s(x)=\tfrac{d}{ds}(b_j(\Lambda(s))-t_{\pi(j)}(\Lambda(s)))\quad\text{if}\quad x\in [b_{j-1}(\Lambda(s)),b_{j}(\Lambda(s))).
\end{equation}
In view of \eqref{eq:defT}, $t_{\pi(j)}(\Lambda(s))-b_j(\Lambda(s))$  measures the displacement between $x$ and $T_sx$ if  $x\in [b_{j-1}(\Lambda(s)),b_{j}(\Lambda(s)))$.
\begin{theorem}[see Theorem~6.2 in \cite{Mi-We}]\label{thm:Mi-Weiss}
{ Let $J\ni s\mapsto \Lambda(s)=(\pi,\lambda(s))\in S_{d}\times\R^d_{>0}$ be a $C^2$-map. For every $s\in J$ denote
by $T_s:I_s\to I_s$ the IET given by $\Lambda(s)$ and let $L_s:I_s\to\R$ be defined by \eqref{eq:defL}.}
Suppose that  for a.e.\ $s\in J$ the IET $T_s:I_s\to I_s$ has no connection and for every $T_{s}$-invariant
measure $\mu$ on $I_s$ we have $\int_{I_s}L_s(x)\,d\mu(x)>0$.  Then  $T_s$ is of recurrence type for a.e.\ $s\in J$.
\end{theorem}

\begin{remark}
In  \cite[Theorem~6.2]{Mi-We} the authors deal with any decaying Federer measure $m$ on the interval $J$ instead of the Lebesgue measure on $J$.
A Borel measure $m$ on $J$ is  decaying and Federer if there are positive constants  $\alpha$, $C$ and $D$ such that for every $x\in \supp(m)$, $0<r,\vep<1$ we have
\[m((x-\vep r,x+\vep r))\leq C\vep^\alpha m((x- r,x+ r)), \ m((x-3 r,x+3r))\leq D m((x-r,x+r)).\]
Since many singular measures are decaying and Federer, the full version of Theorem~6.2 in \cite{Mi-We} gives much more subtle information about the set of all $s\in J$
for which the conclusion of the theorem holds. We should emphasise that all forthcoming results are also true when  ``for a.e.\ $s\in J$'' is replaced by ``for $m$-a.e.\ $s\in J$''.
\end{remark}

\begin{corollary}\label{cor:uniq}
Suppose that for a.e.\ $s\in J$ the translation surface  $(M,\omega(s))$ has no vertical saddle connection and
\begin{align}\label{cond:b-t}
\begin{aligned}
&\tfrac{d}{ds}(b_j(\Lambda(s))-t_{\pi(j)}(\Lambda(s)))\geq 0\quad\text{for every}\quad 1\leq j\leq d\quad\text{with}\\
&\sum_{j=1}^d\tfrac{d}{ds}(b_j(\Lambda(s))-t_{\pi(j)}(\Lambda(s)))> 0.
\end{aligned}
\end{align}
Then  $(M,\omega(s))$ is recurrent for a.e.\ $s\in J$.
\end{corollary}
\begin{proof}
In view of Theorem~\ref{thm:Mi-Weiss} and Remark~\ref{rem:unique},  we only need to show that for a.e.\ $s\in J$ we have $\int_{I_s}L_s(x)\,d\mu(x)>0$
for every $T_{s}$-invariant
measure $\mu$ on $I_s$.

Let $s\in J$ be such that $T_s$ has no saddle connection and  \eqref{cond:b-t} holds. By assumption, a.e.\ $s\in J$ satisfies both conditions.
Since $T_s$ is minimal, the topological support of any $T_s$-invariant measure $\mu$ is $I_s$, i.e.\ the $\mu$-measure of every non-empty open set is positive.
By \eqref{cond:b-t}, $L_s:I_s\to\R$ is a non-negative function which is positive on an open interval. Therefore, $\int_{I_s}L_s(x)\,d\mu(x)>0$.

The final recurrence of Teichm\"uller orbits follows from Remark~\ref{rem:unique}.
\end{proof}

\begin{remark}\label{cond:b-t/l}
Suppose that $\ell:J\to\R_{>0}$ is a $C^\infty$-map. Notice that in Corollary~\ref{cor:uniq} the condition \eqref{cond:b-t} can be replaced by
\begin{align*}
&\frac{d}{ds}\frac{b_j(\Lambda(s))-t_{\pi(j)}(\Lambda(s))}{\ell(s)}\geq 0\quad\text{for every}\quad 1\leq j\leq d \quad\text{ with}\\
&\sum_{j=1}^d\frac{d}{ds}\frac{b_j(\Lambda(s))-t_{\pi(j)}(\Lambda(s))}{\ell(s)}> 0.
\end{align*}
Indeed,  we can rescale all  IETs $T_s:I_s\to I_s$ by dividing the length of each interval $I_s$ by $\ell(s)>0$. Then the rescaled IETs
are affine conjugate to $T_s$ and satisfy the condition \eqref{cond:b-t}.
\end{remark}
For every translation surface $(M,\omega)$ there exists a holomorphic differential on $M$, also denoted by $\omega$, so that $\omega=dz$ in all
local coordinates on $M\setminus\Sigma$ and { $\omega$ vanishes on $\Sigma$}. We also treat $\omega$ as a cohomology element in $H^1(M\setminus\Sigma,\C)$ or $H^1(M,\Sigma,\C)$. We also deal with real cohomology elements
$\rp\omega,\ip\omega\in H^1(M\setminus\Sigma,\R)$ $(H^1(M,\Sigma,\R))$.

We denote by $\langle \,\cdot\,,\,\cdot\, \rangle $ the Kronecker pairing, { i.e.\ $\langle \eta,\gamma \rangle=\int_\gamma\eta$ for $\eta\in H^1$ and $\gamma\in H_1$}.
Suppose that $\gamma$ is a vertical saddle connection of length $s>0$ { on $(M,\omega)$}. Then $\gamma$ can be treated as a relative homology
element in $H_1(M,\Sigma,\Z)$ and  we have
\begin{equation}\label{eq:etarel}
\langle \omega,\gamma\rangle =is,\quad\text{in particular}\quad \langle \rp\omega,\gamma\rangle =0.
\end{equation}
If, for example, $\langle \rp\omega,\gamma\rangle \neq 0$ for every $\gamma\in H_1(M,\Sigma,\Z)$, then
this ensures the absence of vertical saddle connections { on $(M,\omega)$}.

\begin{definition}\label{def:xi}
Suppose that a horizontal interval $I$ in $(M,\omega)$ is a global transversal for the vertical flow $(\varphi^v_t)_{t\in\R}$. Then $T_{\omega,I}=T_\Lambda$ for some $\Lambda=(\pi,\lambda)\in S_d\times\R^d_{>0}$.
For every $1\leq j\leq d$  we denote by $\xi_j=\xi_j(\omega,I)\in H_1(M,\Sigma,\Z)$ the homology class of any loop formed by the vertical orbit segment  starting at any $x\in(b_{j-1}(\Lambda),b_{j}(\Lambda))\subset I$ and ending at $T_\Lambda x\in I$ closed by the segment of $I$ that joins $T_\Lambda x$ and $x$.
\end{definition}
\noindent
Then for every $1\leq j\leq d$ we have
\begin{equation}\label{eq:xibt}
\langle \rp\omega,\xi_j\rangle =b_{j}(\Lambda)-t_{\pi (j)}(\Lambda).
\end{equation}
In view of Corollary~\ref{cor:uniq}, the formula \eqref{eq:xibt} can be useful to prove the recurrence of $(M,\omega)$.

\smallskip

We now give effective formulas to compute $\langle \omega,\gamma\rangle $ for $\gamma\in H_1(M\setminus \Sigma,\Z)$ or $\gamma\in H_1(M,\Sigma,\Z)$ relied on \v{C}ech cohomology.
Suppose that $\mathcal{P}=\{P_\alpha:\alpha\in\mathcal A\}$ is a finite partition of the translation surface $(M,\omega)$ into polygons, i.e.\ $\mathcal{P}=\{P_\alpha:\alpha\in\mathcal A\}$ is a finite family of  closed connected and simply connected subsets of $M$, called polygons, such that
\begin{itemize}
\item[$(i)$] for every $\alpha\in \mathcal A$ there exists a chart $\zeta_\alpha:U_\alpha\to \C$ such that $P_\alpha\setminus\Sigma\subset U_\alpha$, $\zeta_\alpha$ has a { continuous} extension
$\bar{\zeta}_\alpha:U_\alpha\cup P_\alpha\to \C$ { such that $\bar{\zeta}_\alpha:P_\alpha\to \bar{\zeta}_\alpha(P_\alpha)$ is a homeomorphism}, $\bar{\zeta}_\alpha(P_\alpha)$ is a polygon in $\C$ and each point from $\bar{\zeta}_\alpha(P_\alpha\cap \Sigma)$ is its corner;
\item[$(ii)$] if $P_\alpha\cap P_\beta\neq \emptyset$ then it is the union of common sides and corners of the polygons $P_\alpha$, $P_\beta$;
\item[$(iii)$] $\bigcup_{\alpha\in \mathcal A}P_\alpha=M$.
\end{itemize}
We call $\mathcal P$  a partition of $(M,\omega)$ into polygons.
Let
\[\into\mathcal P:=\bigcup_{\alpha\in \mathcal A}\into P_\alpha\quad\text{and}\quad\partial\mathcal P:=\bigcup_{\alpha\in \mathcal A}\partial P_\alpha.\]
We denote by $\operatorname{dir}\mathcal P\subset S^1$ the set of directions of all sides in the partition $\mathcal P$.
\begin{definition}\label{def:1}
Let $\gamma:[a,b]\to M$ be a simple curve (possibly closed) with $\#\gamma([a,b])\cap\partial\mathcal P<+\infty$ and $x\in M$. We define a pairing $\langle x,\gamma\rangle \in \C$ as follows:
\begin{itemize}
\item if $x$ does not belong to the curve then $\langle x,\gamma\rangle :=0$;
\item if $x=\gamma(s_0)$ with $s_0\in (a,b)$ or $x=\gamma(a)=\gamma(b)$ and there exists $\vep>0$ such that $\gamma(s_0,s_0+\vep)\subset \into P_\alpha$ and
$\gamma(s_0-\vep,s_0)\subset \into P_\beta$ then $\langle x,\gamma\rangle :=\bar{\zeta}_\beta(x)-\bar{\zeta}_\alpha(x)$;
\item if $x=\gamma(a)\neq \gamma(b)$ and there exists $\vep>0$ such that $\gamma(a,a+\vep)\subset \into P_\alpha$  then $\langle x,\gamma\rangle :=-\bar{\zeta}_\alpha(x)$;
\item if $x=\gamma(b)\neq \gamma(a)$ and there exists $\vep>0$ such that $\gamma(b-\vep,b)\subset \into P_\beta$  then $\langle x,\gamma\rangle :=\bar{\zeta}_\beta(x)$.
\end{itemize}
\end{definition}
Suppose that the curve $\gamma$  does not start and does not end in $x\in M$.
Notice that if $x\in \into\mathcal P$ then $\langle x,\gamma\rangle =0$. If $x\in \partial\mathcal P\setminus \Sigma$ and $\gamma$
{ passes from $P_\alpha$  to $P_{\beta}$ through $x$} then, by definition,
\begin{equation}\label{eq:<>}
\langle x,\gamma\rangle ={\zeta}_\beta(x)-{\zeta}_\alpha(x)=v_{\alpha,\beta}^C,
\end{equation}
where $C$ is  the connected component of $U_\alpha\cap U_\beta$ containing $x$. { $v_{\alpha,\beta}^C$ is the displacement
of the transfer function between local coordinates (see Definition~\ref{def:transsurf})}.

\begin{theorem}\label{thm:<>}
Suppose that $\gamma:[a,b]\to M$ is a simple curve with $\#\gamma([a,b])\cap\partial\mathcal P<+\infty$ such that
\begin{itemize}
\item[$(i)$] $\gamma(b)=\gamma(a)$ and $\gamma([a,b])\cap\Sigma=\emptyset$ (i.e.\ $[\gamma]\in H_1(M\setminus\Sigma,\Z)$), or
\item[$(ii)$] $\gamma(a),\gamma(b)\in\Sigma$ and $\gamma((a,b))\cap\Sigma=\emptyset$ (i.e.\ $[\gamma]\in H_1(M,\Sigma,\Z)$).
\end{itemize}
Then
\[\langle \omega,[\gamma]\rangle =\sum_{x\in\partial \mathcal P}\langle x,\gamma\rangle .\]
In particular, 
\begin{equation}\label{eq:thetasum}
\langle \rp\omega,[\gamma]\rangle =\sum_{x\in\partial \mathcal P}\rp\langle x,\gamma\rangle.
\end{equation}
\end{theorem}
\begin{proof}
Let $a=t_0<t_1<\ldots<t_{n-1}<t_n=b$ be a partition of $[a,b]$ such that for any $1\leq j\leq n$ there exists $\alpha_j\in\mathcal A$ such that $\gamma(t_{j-1},t_j)\subset \into P_{\alpha_j}$.
Then
\begin{align*}
\langle \omega,[\gamma]\rangle &=\int_\gamma\omega=
\sum_{j=1}^n\int_{t_{j-1}}^{t_j}(\bar{\zeta}_{\alpha_j}\circ\gamma)'(t)\,dt=
\sum_{j=1}^n\big(\bar{\zeta}_{\alpha_j}(\gamma(t_j))-\bar{\zeta}_{\alpha_j}(\gamma(t_{j-1}))\big)\\
&=\sum_{j=1}^{n-1}\big(\bar{\zeta}_{\alpha_j}(\gamma(t_j))-\bar{\zeta}_{\alpha_{j+1}}(\gamma(t_{j}))\big)+
\bar{\zeta}_{\alpha_n}(\gamma(t_n))-\bar{\zeta}_{\alpha_1}(\gamma(t_{0}))
\\
&=\sum_{j=1}^{n-1}\langle \gamma(t_{j}),\gamma \rangle  + \bar{\zeta}_{\alpha_n}(\gamma(b))-\bar{\zeta}_{\alpha_1}(\gamma(a)).
\end{align*}

In the case $(i)$ we can assume that $\gamma(a)=\gamma(b)\in \into\mathcal P$. Then $\alpha_1=\alpha_n$ and $\bar{\zeta}_{\alpha_n}(\gamma(b))=\bar{\zeta}_{\alpha_1}(\gamma(a))$.
Therefore,
\[\langle \omega,[\gamma]\rangle =\sum_{j=1}^{n-1}\langle \gamma(t_{j}),\gamma\rangle =\sum_{x\in\partial \mathcal P}\langle x,\gamma\rangle .\]

In the case $(ii)$ we have $\gamma(a),\gamma(b)\in\Sigma\subset\partial \mathcal P$. As
$\langle \gamma(a),\gamma\rangle =-\bar{\zeta}_{\alpha_1}(\gamma(a))$ and $\langle \gamma(b),\gamma\rangle =\bar{\zeta}_{\alpha_n}(\gamma(b))$,
we have
\[\langle \omega,[\gamma]\rangle =\sum_{j=0}^{n}\langle \gamma(t_{j}),\gamma\rangle =\sum_{x\in\partial \mathcal P}\langle x,\gamma\rangle,\]
which completes the proof.
\end{proof}

\begin{definition}\label{def:DBE}
Suppose that  the vertical direction does not belong to $\operatorname{dir}\mathcal P$. Denote by:
\begin{itemize}
 \item $\widehat{D}=\widehat{D}(\omega,\mathcal P)$ the set of triples $(\alpha,\beta,C)$ ({ $\alpha,\beta\in\mathcal A$ and} $C$ is a connected component of $U_\alpha\cap U_\beta$) for which there is a vertical orbit segment $\{\varphi^v_tx:t\in[-\vep,\vep]\}\subset C$ ($\vep>0$) such that $\varphi^v_tx\in P_\beta$ for $t\in[-\vep,0]$ and $\varphi^v_tx\in P_\alpha$ for $t\in[0,\vep]$;
 \item ${D}={D}(\omega,\mathcal P)$ the subset of all triples $(\alpha,\beta, C)\in\widehat{D}$ such that the point $x$ belongs to the interior of a common side of $P_\alpha$ and $P_\beta$;
 \item $B=B(\omega,\mathcal P)$ the set of pairs $(\sigma,\alpha)\in\Sigma\times \mathcal A$ for which $\sigma\in P_\alpha$ and there is a vertical curve $\gamma:[0,\vep]\to P_\alpha$ with $\gamma(0)=\sigma$ and $\bar{\zeta}_\alpha(\gamma(t))=\bar{\zeta}_\alpha(\sigma)+it$ for $t\in[0,\vep]$.
 \item $E=E(\omega,\mathcal P)$ the set of pairs $(\sigma,\beta)\in\Sigma\times \mathcal A$ for which $\sigma\in P_\beta$ and there is a vertical curve $\gamma:[-\vep,0]\to P_\beta$ with $\gamma(0)=\sigma$ and $\bar{\zeta}_\beta(\gamma(t))=\bar{\zeta}_\beta(\sigma)+it$ for $t\in[-\vep,0]$.
\end{itemize}
\end{definition}

\begin{remark}\label{DBEtheta}
Suppose that $\mathcal P$ is a partition of $(M,\omega)$ into polygons  and $\theta\notin \operatorname{dir}\mathcal P$. In a similar way as in Definition~\ref{def:DBE} we can define
the sets $D_\theta$, $B_\theta$ and $E_\theta$ using orbits in direction $\theta$ instead of vertical orbits.

Let us consider the rotated translation surface $(M,r_{\pi/2-\theta}\omega)$. Then the rotated partition $r_{\pi/2-\theta}\mathcal P$ is a partition of  $(M,r_{\pi/2-\theta}\omega)$
into polygons such that  the vertical directions does not belong to $\operatorname{dir}r_{\pi/2-\theta}\mathcal P$ and
\begin{align*}
{D}(r_{\pi/2-\theta}\omega,&r_{\pi/2-\theta}\mathcal P)={D}_\theta(\omega,\mathcal P),\ {B}(r_{\pi/2-\theta}\omega,r_{\pi/2-\theta}\mathcal P)={B}_\theta(\omega,\mathcal P),\\ &{E}(r_{\pi/2-\theta}\omega,r_{\pi/2-\theta}\mathcal P)={E}_\theta(\omega,\mathcal P).
\end{align*}
\end{remark}

\begin{lemma}\label{lem:cech}
For every  $(\alpha,\beta,C)\in \widehat{D}\setminus {D}$ there exists a sequence $\{(\alpha_{j+1},\alpha_{j}, C_j)\}_{j=1}^n$ of elements in ${D}$ such that $\alpha_1=\beta$, $\alpha_{n+1}=\alpha$ and
\[v_{\alpha,\beta}^C=\sum_{j=1}^nv_{\alpha_{j+1},\alpha_{j}}^{C_j}.\]
\end{lemma}

\begin{proof}
Suppose that $(\alpha,\beta,C)\in \widehat{D}\setminus {D}$. Then there is  a vertical  orbit segment $\{\varphi^v_tx:t\in[-\vep,\vep]\}\subset C$ such that $\varphi^v_t x\in \into P_\beta$ for $t\in[-\vep,0)$,  $\varphi^v_tx\in \into P_\alpha$ for $t\in(0,\vep]$ and $x$ is a common corner of $P_\alpha$ and $P_\beta$. Then there exists $\delta>0$ such that
\begin{itemize}
\item the rectangle $R:=\{\varphi^v_t\varphi^h_s x:s\in[0,\delta],t\in[-\vep,\vep]\}\subset M\setminus \Sigma$ is well defined;
\item $R$ does not contain any corner of the partition $\mathcal P$ other then $x$;
\item $\varphi^v_{-\vep} \varphi^h_s x\in \into P_\beta$  and $\varphi^v_\vep\varphi^h_s x\in \into P_\alpha$ for $s\in[0,\delta]$.
\end{itemize}
Let us consider the vertical orbit segment $\{\varphi^v_t\varphi^h_\delta x:t\in[-\vep,\vep]\}$. It has finitely many intersection
points with $\partial\mathcal P$ and all of them are not corners. Let
\[-\vep=t_0<t_1<\ldots<t_{n}<t_{n+1}=\vep\ \text{ and }\
\beta=\alpha_1,\alpha_2,\ldots,\alpha_{n},\alpha_{n+1}=\alpha\]
be elements of $\mathcal A$ such that
\[\varphi^v_t\varphi^h_\delta x\in P_{\alpha_j}\quad\text{for all}\quad t\in[t_{j-1},t_j]\ \text{ and }\ 1\leq j\leq n+1.\]
For every  $1\leq j\leq n$ denote by $C_j$ the connected component of $U_{\alpha_j}\cap U_{\alpha_{j+1}}$ that contains $\varphi^v_{t_j}\varphi^h_\delta x\in P_{\alpha_j}\cap P_{\alpha_{j+1}}$. Then $(\alpha_{j+1},\alpha_{j},C_j)\in {D}$ for every $1\leq j\leq n$.

Let $\gamma$ be the boundary (oriented) of the rectangle $R$. As $[\gamma]\in H_1(M\setminus\Sigma,\R)$ is the zero homology element,
we have $\langle \omega,[\gamma]\rangle =0$. Therefore, by Theorem~\ref{thm:<>} and \eqref{eq:<>}, we have
\[0=\langle \omega,[\gamma]\rangle =\langle x,\gamma\rangle +\sum_{j=1}^n\langle \varphi^v_{t_j}\varphi^h_\delta x,\gamma\rangle =-v_{\alpha,\beta}^C+\sum_{j=1}^nv_{\alpha_{j+1},\alpha_{j}}^{C_j}.\]
which completes the proof.
\end{proof}

Let us come back to a $C^\infty$-curve  $J\ni s\mapsto\omega(s)\in\mathcal M(M,\Sigma)$.
Suppose that there is a finite open cover $(U_\alpha)_{\alpha\in\mathcal A}$ of $M\setminus\Sigma$ and for every $s\in J$
{ there exists}
a partition $\mathcal P(s)=\{P_\alpha(s):\alpha\in \mathcal A\}$ of $(M,\omega(s))$ into polygons so that $ P_{\alpha}(s)\setminus\Sigma\subset U_\alpha$ for every $\alpha\in \mathcal A$ and $s\in J$. Moreover,
assume that for every $\alpha\in \mathcal A$ the polygon $P_\alpha(s)$ and the corresponding chart $\zeta^s_\alpha:U_\alpha\to\C$
vary $C^\infty$-smoothly with $s\in J$. It follows that for every  connected component $C$ of $U_\alpha\cap U_\beta$ the
 map \[J\ni s\mapsto v^C_{\alpha,\beta}(s)={ \zeta^s_\beta(x)-\zeta^s_\alpha(x)}\in \C\]
{ ($x$ is any element of $C$)} is of class $C^\infty$.

Assume that the vertical direction does not belong to  $\operatorname{dir}\mathcal P(s)$ for all $s\in J$. Then the sets $D(\omega(s),\mathcal P(s))$, $B(\omega(s),\mathcal P(s))$, $E(\omega(s),\mathcal P(s))$ (see Definition~\ref{def:DBE}) do not depend on $s\in J$. Let us consider three finite subsets of $C^\infty(J,\C)$:
\begin{align*}
\mathscr D &:=\{s\mapsto v^C_{\alpha,\beta}(s):(\alpha,\beta,C)\in D\},\\
\mathscr B &:=\{s\mapsto-\bar{\zeta}^s_{\alpha}(\sigma):(\sigma,\alpha)\in B\},\\
\mathscr E &:=\{s\mapsto \bar{\zeta}^s_{\beta}(\sigma):(\sigma,\beta)\in E\}.
\end{align*}
{ If $s\in J$ is fixed, then
\begin{itemize}
\item the values of functions from the family $\mathscr D$ at the point $s$ indicate all displacements
between local coordinates when a vertical upward curve passes between polygons (through a common side);
\item the values of functions from the family $\mathscr B$ at the point $s$ indicate all opposite local coordinates
of singular points when a vertical upward curve starts at a singularity;
\item the values of functions from the family $\mathscr E$ at the point $s$ indicate all local coordinates
 of singular points when a vertical downward curve starts at a singularity.
\end{itemize}
}
For any pair of $C^1$-maps $f,g:J\to\R$ we define their bracket $[f,g]:J\to\R$ by $[f,g](s)=f'(s)g(s)-f(s)g'(s)$ for $s\in J$.

\begin{theorem}\label{thm:gencrit}
Let $\ell:J\to\R_{>0}$ be a $C^\infty$-map. Suppose that
\begin{itemize}
\item[$(i)$] for any $f\in \mathscr B$, $g\in \mathscr E$
and any sequence $(n_h)_{h\in \mathscr D}$ of numbers in $\Z_{\geq 0}$ such that the map $f+g+\sum_{h\in \mathscr D}n_hh$ is non-zero we have
\[\rp f(s)+\rp g(s)+\sum_{h\in \mathscr D} n_h\rp h(s)\neq 0\ \text{ for a.e. }\ s\in J;\]
\item[$(ii_+)$] $[\rp h,\ell](s)\geq 0$ for all $h\in  \mathscr D$ and $s\in J$  with
\[\sum_{h\in  \mathscr D}[\rp h,\ell](s)> 0 \quad\text{for a.e. $s\in J$, or}\]
\item[$(ii_-)$] $[\rp h,\ell](s)\leq 0$ for all $h\in  \mathscr D$ and $s\in J$  with
\[\sum_{h\in  \mathscr D}[\rp h,\ell](s)< 0\quad\text{for a.e.}\ s\in J.\]
\end{itemize}
Then $(M,\omega(s))$ is recurrent for a.e.\ $s\in J$.
\end{theorem}

\begin{proof}
The proof relies on using Corollary~\ref{cor:uniq} combined with Remark \ref{cond:b-t/l}.
First we show that for a.e.\ $s\in J$ there is no vertical saddle connection on $(M,\omega(s))$. By assumption $(i)$, there exists
a subset $J_0\subset J$ of full Lebesgue measure such that for every $s\in J_0$ if  $f\in \mathscr B$, $g\in \mathscr E$ and $(n_h)_{h\in \mathscr D}$ is a sequence of numbers in $\Z_{\geq 0}$ with $f+g+\sum_{h\in \mathscr D}n_hh$ being non-zero map, then
\begin{equation}\label{neq:f+g+h}
 \rp f(s)+\rp g(s)+\sum_{h\in \mathscr D}n_h\rp h(s)\neq 0.
\end{equation}

 We show that for every $s\in J_0$ there is no vertical saddle connection on $(M,\omega(s))$. Indeed, suppose contrary to our claim
that $\gamma:[0,\tau]\to M$ is a vertical saddle connection on $(M,\omega(s))$, i.e.\ $\gamma(0),\gamma(\tau)\in\Sigma$,
$\gamma(0,\tau)\subset M\setminus \Sigma$ and $\gamma'(t)=i$ for $t\in(0,\tau)$ in local coordinates on $(M,\omega(s))$.  Then,  by  Theorem~\ref{thm:<>},
\begin{equation}\label{eq:tau}
i\tau=\langle \omega(s),[\gamma]\rangle =\sum_{x\in\partial \mathcal P(s)}\langle x,\gamma\rangle .
\end{equation}
Let
\[0=t_0<t_1<\ldots<t_{n}<t_{n+1}=\tau\ \text{ and }\
\alpha_1,\alpha_2,\ldots,\alpha_{n},\alpha_{n+1}\in\mathcal{A}\]
be such that
\[\gamma(t)\in \into P_{\alpha_j}(s)\quad\text{for all}\quad t\in(t_{j-1},t_j)\ \text{ and }\ 1\leq j\leq n+1.\]
For every $1\leq j\leq n$ denote by $C_j$ the connected component of $U_{\alpha_j}\cap U_{\alpha_{j+1}}$ that contains
$\gamma(t_j)\in P_{\alpha_j}\cap P_{\alpha_{j+1}}$. Then
\[(\gamma(0),\alpha_0)\in B,\quad (\gamma(\tau),\alpha_{n+1})\in E, \quad(\alpha_{j+1},\alpha_{j},C_j)\in \widehat{D}\]
for every $1\leq j\leq n$ and
\[\sum_{x\in\partial \mathcal P(s)}\langle x,\gamma\rangle =\sum_{j=0}^{n+1}\langle \gamma(t_j),\gamma\rangle =
-\bar{\zeta}^s_{\alpha_0}(\gamma(0))+\sum_{j=1}^nv_{\alpha_{j+1},\alpha_j}^{C_j}(s)+\bar{\zeta}^s_{\alpha_{n+1}}(\gamma(\tau)).\]
In view of Lemma~\ref{lem:cech}, \eqref{eq:tau} and by the definition of sets $\mathscr D$, $\mathscr B$, $\mathscr E$,
there exist $f\in \mathscr B$, $g\in \mathscr E$  and a sequence $(n_h)_{h\in \mathscr D}$ of numbers in $\Z_{\geq 0}$ such that
\begin{equation}\label{eq:neq0}
i\tau=\sum_{x\in\partial \mathcal P(s)}\langle x,\gamma\rangle =f(s)+g(s)+\sum_{h\in \mathscr D}n_hh(s).
\end{equation}
Hence
\begin{equation}\label{eq:zero}
0=\rp\langle \omega(s),[\gamma]\rangle
=\rp f(s)+\rp g(s)+\sum_{h\in \mathscr D}n_h\rp h(s).
\end{equation}
By \eqref{eq:neq0}, the map $f+g+\sum_{h\in \mathscr D}n_hh$ is non-zero. As $s\in J_0$,   \eqref{eq:zero} contradicts \eqref{neq:f+g+h}. This yields the absence of vertical  saddle connections.

\smallskip
Fix  $\alpha_0\in\mathcal A$ and for every $s\in J$ choose a vertical interval $I_s\subset \into P_{\alpha_0}(s)$ so that the map $s\mapsto I_s$ is of class $C^\infty$. Then for every $s_0\in J_0$ (the subset $J_0\subset J$ is defined in the previous paragraph) the interval $I_{s_0}$ is a global transversal for the vertical flow on
$(M,\omega(s_0))$.
Since $s\mapsto (M,\omega(s))$ is a $C^\infty$-curve in $\mathcal{M}(M,\Sigma)$ and the choice of the interval $I_s$ in $(M,\omega(s))$ is smooth, for every $s_0\in J_0$ there exists $\vep>0$
such that for every $s\in(s_0-\vep,s_0+\vep)$ the interval $I_s$ is a global transversal for the vertical flow on
$(M,\omega(s))$ and the corresponding first return map $T_{\omega(s),I_s}=T_s:I_s\to I_s$ has the same combinatorial data (the number of exchanged intervals and permutation) as $T_{s_0}$. It follows that there exists a countable family $\mathcal{J}$ of pairwise disjoint open subintervals in $J$ such that:
\begin{itemize}
\item[(\emph{i})] the complement of $\bigcup_{\Delta\in\mathcal{J}}\Delta$ in $J$ has zero Lebesgue measure;
\item[(\emph{ii})] $I_s\subset \into P_{\alpha_0}(s)$ is a global transversal for every  $s\in \bigcup_{\Delta\in\mathcal{J}}\Delta$;
\item[(\emph{iii})] for every $\Delta\in\mathcal{J}$ all IETs $T_s$, $s\in \Delta$ have the same combinatorial data.
\end{itemize}
Therefore, it suffices to show that for every $\Delta\in\mathcal{J}$ and for a.e.\ $s\in \Delta$ the translation surface
$(M,\omega(s))$ is recurrent.

Fix $\Delta\in\mathcal{J}$. Then there exist $d\geq 2$, { a permutation $\pi\in S_d$}  and a $C^\infty$-map $\Delta\ni s \mapsto \Lambda(s)=(\pi,\lambda(s))\in S_d\times \R^d_{>0}$ such that $T_s=T_{\Lambda(s)}$ for all $s\in \Delta$. In view of Corollary~\ref{cor:uniq} combined with Remark~\ref{cond:b-t/l}, we need to show that for a.e.\ $s\in \Delta$  we have
\begin{align}
\label{eq:fin}
&[b_j(\Lambda(s))-t_{\pi(j)}(\Lambda(s)),\ell(s)]\geq 0\ \text{ for }\ 1\leq j\leq d\ \text{ and }\\
\label{eq:fin1}
&\sum_{j=1}^d[b_j(\Lambda(s))-t_{\pi(j)}(\Lambda(s)),\ell(s)]>0 .
\end{align}

 For every $s\in \Delta$ and $1\leq j\leq d$ let $\xi_j(s)=\xi_j(\omega(s),I_s)\in H_1(M,\Sigma,\Z)$ be the homology element defined in Definition~\ref{def:xi}. Then, by \eqref{eq:xibt},
\begin{equation*}
\rp \langle \omega(s),\xi_j(s)\rangle =b_{j}(\Lambda(s))-t_{\pi (j)}(\Lambda(s))\ \text{ for every }\ s\in\Delta.
\end{equation*}
Since $\xi_j(s)$ is the homology class of a loop $\gamma^j_s$ formed by the segment of the vertical orbit in  $(M,\omega(s))$ starting at any $x^j_s\in(b_{j-1}(\Lambda(s)),b_{j}(\Lambda(s)))\subset I_s$ and ending at $T_sx^j_s\in I_s$ closed by the segment of $I_s\subset \into P_{\alpha_0}(s)$ that joins $T_sx^j_s$ and $x^j_s$, by Theorem~\ref{thm:<>} for every $s\in \Delta$ we have
\begin{align}\label{eq:omxi}
\begin{aligned}
\rp\langle \omega(s),\xi_j(s)\rangle &=\sum_{x\in\partial \mathcal P(s)}\rp\langle x,\gamma^j_s\rangle \\
&=\sum_{(\alpha,\beta,C)\in \widehat{D}}n^j_{(\alpha,\beta,C)}(s)\rp v_{\alpha,\beta}^C(s),
\end{aligned}
\end{align}
where $n^j_{(\alpha,\beta,C)}(s)$ is the number of meeting points $x\in \partial P(s)$ of $\gamma^j_s$ with $\partial P(s)$ such that $\gamma^j_s$ passes from $\into P_\beta(s)$ to $\into P_\alpha(s)$ through $x$
and $x$ belongs to the connected component $C$ of $U_\alpha\cap U_\beta$.

Take any $s_0\in\Delta$. Since the partition $\mathcal P(s_0)$ into polygons has finitely many corners, for every $1\leq j\leq d$ we can find $x^j_{s_0}\in (b_{j-1}(\Lambda(s_0)),b_{j}(\Lambda(s_0)))$ such that the corresponding loops $\gamma^j_{s_0}$, $1\leq j\leq d$
do not meet the corners of $\mathcal P(s_0)$. Then we choose other points $x^j_s$ for $s\in\Delta\setminus\{s_0\}$ such that the map
\[\Delta\ni s\mapsto x^j_s\in (b_{j-1}(\Lambda(s)),b_{j}(\Lambda(s)))\] is of class $C^\infty$ for  every $1\leq j\leq d$.
We deal with the family of corresponding loops $\gamma^j_s$ for  $s\in\Delta$ and $1\leq j\leq d$. By the continuity of the maps $s\mapsto\gamma_s^j$, we can find $\vep>0$ such that $(s_0-\vep,s_0+\vep)\subset\Delta$ and for every $s\in (s_0-\vep,s_0+\vep)$ the loops $\gamma^j_s$, $1\leq j\leq d$ do not meet the corners of $\mathcal P(s)$. It follows that each map $n^j_{(\alpha,\beta,C)}$ is constant on $(s_0-\vep,s_0+\vep)$ and the range of the second sum in \eqref{eq:omxi} is $D$.
Therefore for every $1\leq j\leq d$ there exists a sequence $(n^j_h)_{h\in \mathcal D}$ numbers in $\Z_{\geq 0}$ such that
\begin{align}\label{eq:btsin}
b_{j}(\Lambda(s))-t_{\pi (j)}(\Lambda(s))=\rp\langle \omega(s),\xi_j(s)\rangle =\sum_{h\in \mathcal D}n^j_h\rp h(s)
\end{align}
for all $s\in (s_0-\vep,s_0+\vep)$. It follows that
\begin{equation}\label{eq:btbarcket}
[b_{j}(\Lambda(s))-t_{\pi (j)}(\Lambda(s)),\ell(s)]=\sum_{h\in \mathcal D}n^j_h[\rp h(s),\ell(s)]
\end{equation}
for all $s\in (s_0-\vep,s_0+\vep)$ and $1\leq j\leq d$.


Now assume that the condition $(ii_+)$ holds.
In view of $(ii_+)$, $[\rp h(s),\ell(s)]\geq 0$ for all  $h\in \mathcal D$ and $s\in (s_0-\vep,s_0+\vep)$. Therefore, by \eqref{eq:btbarcket}, $[b_{j}(\Lambda(s))-t_{\pi (j)}(\Lambda(s)),\ell(s)]\geq 0$ for all $s\in(s_0-\vep,s_0+\vep)$ and $1\leq j\leq d$. As $s_0$ is an arbitrary element of $\Delta$,
it follows that  $[b_{j}(\Lambda(s))-t_{\pi (j)}(\Lambda(s)),\ell(s)]\geq 0$ for every $s\in\Delta$. It gives \eqref{eq:fin} under the assumption $(ii_+)$.
%

To complete the proof in this case we need to show \eqref{eq:fin1}.
Suppose, contrary to our claim, that the subset $J_1\subset J$ of all  $s\in J$ such that
\begin{equation*}
[b_{j}(\Lambda(s))-t_{\pi (j)}(\Lambda(s)),\ell(s)]= 0\  \text{ for every }\ 1\leq j\leq d
\end{equation*}
has positive Lebesgue measure.
By the assumption $(ii_+)$, there exist $s_0\in J_0\cap J_1$ and $h_0\in\mathcal D$ such that
\begin{align}\label{eq:h_0}
[\rp h_0(s_0),\ell(s_0)]>0\ \text{ and }\
[\rp h(s_0),\ell(s_0)]\geq 0\ \text{for all}\ h\in \mathcal D.
\end{align}
Let $(\alpha_0,\beta_0,C_0)\in D$ be a triple such that $h_0(s)=v_{\alpha_0,\beta_0}^{C_0}(s)$ for all $s\in J$. For every $s\in J$ choose a common side $e^0_s$  (without the ends) of $P_{\alpha_0}(s)$ and $P_{\beta_0}(s)$ contained in $C_0$ so that the map $s\mapsto e^0_s$ is smooth.

Since the flow $(\varphi^v_t)_{t\in\R}$ on $(M,\omega(s_0))$ is minimal, there exists $1\leq j\leq d$ and $x^j_{s_0}\in (b_{j-1}(\Lambda(s_0)),b_{j}(\Lambda(s_0)))$ such that the corresponding loop $\gamma^j_{s_0}$ has an intersection with $e^0_{s_0}$ and
does not meet any corner of $\mathcal P(s_0)$. As in the first part of the proof, there exists $\vep>0$ and a smooth map $(s_0-\vep,s_0+\vep)\ni s\mapsto x^j_s\in I_s$ such that for every $s\in(s_0-\vep,s_0+\vep)$ the corresponding loop $\gamma^j_{s}$
does not meet any corner of $\mathcal P(s)$. It follows that the map $n^j_{(\alpha_0,\beta_0,C_0)}$ is constant on $(s_0-\vep,s_0+\vep)$ and takes a positive value.  Therefore
there exists a sequence $(n_h)_{h\in \mathcal D}$ of numbers in $\Z_{\geq 0}$ such that $n_{h_0}>0$ and
\begin{align*}
b_{j}(\Lambda(s))-t_{\pi (j)}(\Lambda(s))=\rp\langle \omega(s),\xi_j(s)\rangle =\sum_{h\in \mathcal D}n_h\rp h(s)
\end{align*}
for all $s\in (s_0-\vep,s_0+\vep)$.
In view of \eqref{eq:h_0}, it follows that
\begin{align*}
[b_{j}(\Lambda(s_0))-t_{\pi (j)}(\Lambda(s_0)),\ell(s_0)]&=\sum_{h\in \mathcal D}n_h[\rp h(s_0),\ell(s_0)]\\
&\geq n_{h_0}[\rp h_0(s_0),\ell(s_0)]>0.
\end{align*}
This contradicts the fact that $s_0\in J_1$ and finishes the proof of \eqref{eq:fin1}. This completes the proof  under the assumption $(ii_+)$.

\smallskip
Now assume that the condition $(ii_-)$ holds. Then we consider the reverse curve $-J\ni s\mapsto (M,\omega(-s))\in \mathcal M(M,\Sigma)$. Since the reverse curve satisfies
$(i)$ and $(ii_+)$, the assertion of the theorem follows from the previous part of the proof.
\end{proof}

\section{Billiards on tables with vertical and horizontal sides}\label{sec:poly}
In the next two sections we deal with  the billiard flow in directions $\pm\pi/4$, $\pm 3\pi/4$ on tables
with vertical and horizontal sides. More precisely, we consider smooth curves of such tables.
The aim of this part of the paper is to formulate and prove a criterion (Theorem~\ref{thm:mainsurf}) for unique ergodicity of the billiard
flow on almost every table in the curve. The proof of Theorem~\ref{thm:mainsurf}
relies on Theorem~\ref{thm:gencrit}.

\smallskip

Denote by $\Xi$ the set of  sequences $(\overline{x},\overline{y})=(x_i,y_i)_{i=1}^k$ of points in $\R^2_{>0}$ such that
\[0<x_1<x_2<\ldots<x_{k-1}<x_k\quad\text{and}\quad 0<y_k<y_{k-1}<\ldots<y_2<y_1.\]
\begin{figure}[h]
\includegraphics[width=0.6 \textwidth]{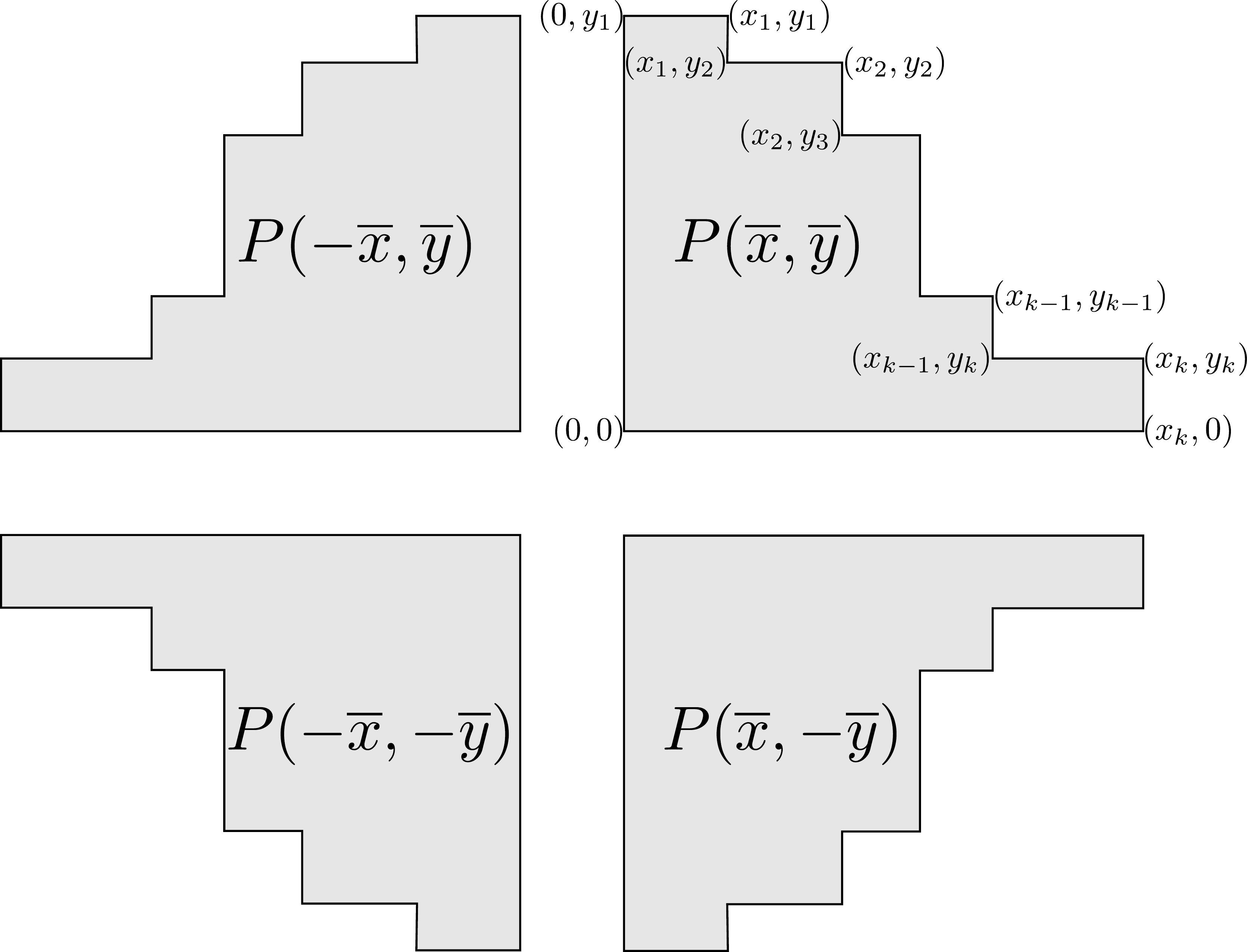}
\caption{Basic polygons $P(\overline{x},\overline{y})$, $P(-\overline{x},\overline{y})$, $P(\overline{x},-\overline{y})$, $P(-\overline{x},-\overline{y})$.}\label{fig:basicpolygon}
\end{figure}
Let $k(\overline{x},\overline{y}):=k$.
For every $(\overline{x},\overline{y})\in\Xi$
denote by $P(\overline{x},\overline{y})$ the right-angle { staircase} polygon on $\R^2$ (i.e.\ with angles $\pi/2$ or $3\pi/2$) with consecutive vertices:
\[(0,0),(0,y_1),(x_1,y_1),(x_1,y_2),\ldots,(x_{k-1},y_{k-1}),(x_{k-1},y_{k}),(x_k,y_k),(x_k,0),\]
see Figure~\ref{fig:basicpolygon}.

Denote by $\Gamma$ the four element group
generated by the vertical and the horizontal reflections $\gamma_v, \gamma_h:\R^2\to\R^2$. We extent the action of $\Gamma$ to the space of finite sequences of points in $\R^2$.
The polygons of the form
\begin{align*}
P(-\overline{x},\overline{y})&= P(\gamma_v(\overline{x},\overline{y})):=\gamma_v P(\overline{x},\overline{y}),\\
P(\overline{x},-\overline{y})&=P(\gamma_h (\overline{x},\overline{y})):=\gamma_h P(\overline{x},\overline{y}),\\
P(-\overline{x},-\overline{y})&= P(\gamma_v\circ\gamma_h(\overline{x},\overline{y})):=\gamma_v\circ\gamma_h P(\overline{x},\overline{y})
\end{align*}
are called \emph{basic polygons}, see Figure~\ref{fig:basicpolygon}.
We say that:
\begin{itemize}
\item $\gamma[(0,0),(0,y_1)]$ is the long vertical side;
\item $\gamma[(0,0),(x_k,0)]$ is the long horizontal side;
\item $\gamma[(x_k,0),(x_k,y_k)]$ is the short vertical side;
\item $\gamma[(0,y_1),(x_1,y_1)]$ is the short horizontal side
\end{itemize}
of the basic polygon $P(\gamma(\overline{x},\overline{y}))$ { for $\gamma\in\Gamma$}.

\smallskip

We deal with billiard flows on right angle connected generalized polygons which are the union of finitely many basic polygons $P(\gamma(\overline{x},\overline{y}))$ for $(\overline{x},\overline{y})\in\Xi$, $\gamma\in\Gamma$ so that some sides of basic polygons are glued by translations.

Denote by $\mathscr{P}$ the collection of such generalized polygons for which the sides of the basic polygons
can be glued only in the following four cases:
\begin{itemize}
\item[$(V)$] we can glue $P(\overline{x},\pm\overline{y})$ with $P(-\overline{x}',\pm\overline{y}')$ along the long vertical sides if their lengths are the same;
\item[$(H)$] we can glue $P(\pm\overline{x},\overline{y})$ with $P(\pm\overline{x}',-\overline{y}')$  along the long horizontal sides if their lengths are the same;
\item[$(v)$] we can glue $P(\overline{x},\pm\overline{y})$ with $P(-\overline{x}',\pm\overline{y}')$ along the short vertical sides if their lengths are the same;
\item[$(h)$] we can glue $P(\pm\overline{x},\overline{y})$ with $P(\pm\overline{x}',-\overline{y}')$ along the short horizontal sides if their lengths are the same.
\end{itemize}

Notice that a generalized polygon $\mathbf{P}\in\mathscr{P}$ is not necessary a polygon in $\R^2$, $\mathbf P$ should be treated rather as a translation surface with boundary.
Translation surface of this type, called parking garages, and the corresponding billiard flows were already studied in \cite{Co-We} in the context of Veech dichotomy.

\smallskip

Suppose that a generalized polygon $\mathbf{P}\in\mathscr{P}$ is formed by gluing $M\geq 1$ basic polygons $\{\gamma_m P_m(\overline{x}^m,\overline{y}^m):m\in\mathcal I\}$, { where} $\mathcal I$ is an $M$-element set of indices
of basic polygons { and elements $\{\gamma_{m}:m\in\mathcal I\}$ in $\Gamma$ describe their types}.  Then we write
\[\mathbf{P}=\biguplus_{m\in \mathcal{I}}\gamma_m P_m(\overline{x}^m,\overline{y}^m).\]
We label the $m$-th basic polygon $P(\gamma_m(\overline{x}^m,\overline{y}^m))$ by the additional subscript $m$ because
many copies of the polygon $P(\gamma_m(\overline{x}^m,\overline{y}^m))$ can be used  to create the generalized polygon $\mathbf P$. The additional subscript helps us to distinguish them.

\begin{figure}[h]
\includegraphics[width=0.7 \textwidth]{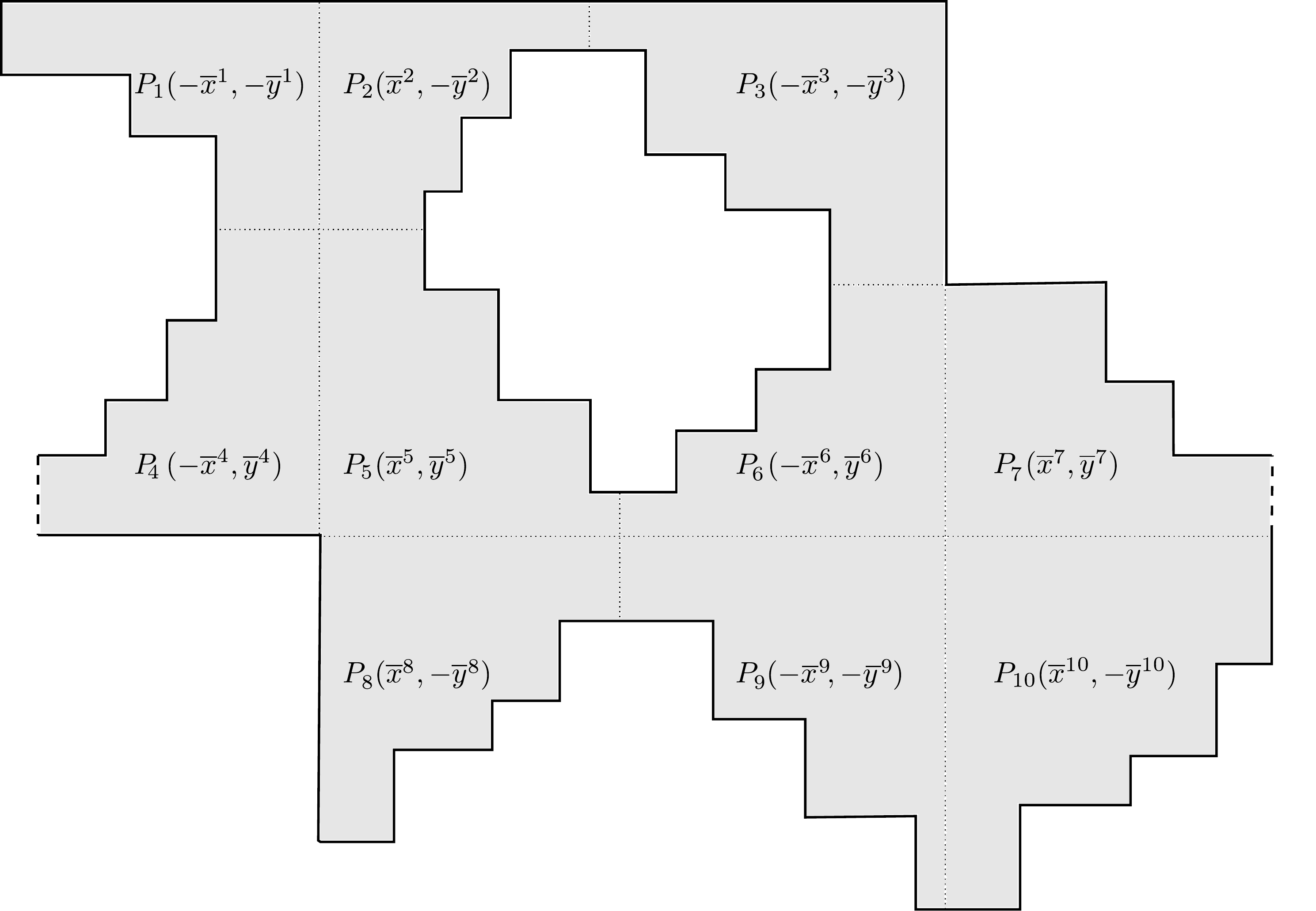}
\caption{An example of a generalized polygon $\mathbf{P}\in\mathscr{P}$ with
$1\sim_V 2$, $4\sim_V 5$, $6\sim_V 7$, $9\sim_V 10$,
$5\sim_H 8$, $6\sim_H 9$, $7\sim_H 10$,
$2\sim_v 3$, $5\sim_v 6$,$7\sim_v 4$, $8\sim_v 9$,
$1\sim_h 4$, $2\sim_h 5$, $3\sim_h 6$.}
\label{fig:polygonP}
\end{figure}

In order to describe the generalized polygon $\mathbf P$ fully we define four symmetric relations $\sim_V$, $\sim_H$, $\sim_v$, $\sim_h$ on $\mathcal I$ which reflect the gluing rules of
basic polygons in $\mathbf{P}$. We let $m\sim_a m'$ if
the polygons  $\gamma_m P_{m}(\overline{x}^m,\overline{y}^m)$ and $\gamma_{m'} P_{m'}(\overline{x}^{m'},\overline{y}^{m'})$
are glued in accordance with the scenario $(a)$, for $a=V,H,v,h$. Moreover, if $m\in\mathcal{I}$ is not $\sim_a$-related to any other element of $\mathcal{I}$
then we adopt the convention that $m\sim_a m$.

For every $m\in\mathcal I$ let $k_m:=k(\overline{x}^m,\overline{y}^m)$. We refer to the collection
\[\big\{\mathcal I,\, (\gamma_m)_{m\in\mathcal I},\, (k_m)_{m\in\mathcal I},\ \sim_V,\ \sim_H,\ \sim_v,\ \sim_h\big\}\]
as the \emph{combinatorial data} of the generalized polygon $\mathbf P\in\mathscr P$.


\begin{remark}\label{rem:gluing_rules}
Suppose that $\mathbf{P}\in\mathscr{P}$ and  consider the directional billiard flow on $\mathbf{P}$ in  directions $\Gamma(\pi/4)=\{\pm\pi/4,\pm 3\pi/4\}$. After performing the unfolding procedure emulating the standard procedure (for rational polygons on $\R^2$) coming from \cite{Fox-Ker} and \cite{Ka-Ze}
{ (roughly presented at the end of Section~\ref{sec:change})} we obtain a translation surface $M(\mathbf{P})\in\mathcal M$.
 Then the billiard flow is isomorphic to the directional flow $(\varphi^{\pi/4}_t)_{t\in\R}$ on $M(\mathbf{P})$.
{ Recall that $M(\mathbf{P})$ is glued from four transformed copies of $\mathbf{P}$, i.e.\ $\mathbf{P}$, $\gamma_v\mathbf{P}$, $\gamma_h\mathbf{P}$ and $\gamma_v\circ\gamma_h\mathbf{P}$.}
Therefore, the translation surface $M(\mathbf{P})$ has a natural partition into basic polygons
\begin{equation}\label{eq:MP}
\mathcal{P}_{\mathbf{P}}=\{P_m(\gamma(\overline{x}^m,\overline{y}^m)):(m,\gamma)\in \mathcal{I}\times\Gamma\}.
\end{equation}
\textbf{ Gluing rules.} The sides of the basic polygons in $M(\mathbf{P})$ are identified by translations in the following way:
\begin{itemize}
\item[$(i)$] for $1\leq i< k_m$ the vertical side $[(x^m_i,\pm y^m_{i}),(x^m_i,\pm y^m_{i+1})]$ of $P_m(\overline{x}^m,\pm\overline{y}^m)$ is identified with the side $[(-x^m_i,\pm y^m_{i}),(-x^m_i,\pm y^m_{i+1})]$ of $P_m(-\overline{x}^m,\pm\overline{y}^m)$;
\item[$(ii)$] for $1< i\leq k_m$ the horizontal side  $[(\pm x^m_{i-1}, y^m_{i}),(\pm x^m_{i}, y^m_{i})]$ of $P_m(\pm\overline{x}^m,\overline{y}^m)$ is identified with the side  $[(\pm x^m_{i-1}, -y^m_{i}),(\pm x^m_i, -y^m_{i})]$ of $P_m(\pm\overline{x}^m,-\overline{y}^m)$;
\item[$(iii)$] if $m\sim_V m'$ then the long vertical side
 of $P_m(\overline{x}^{m},\pm\overline{y}^m)$ is identified with the long vertical side  of $P_{m'}(-\overline{x}^{m'},\pm\overline{y}^{m'})$;
\item[$(iv)$] if $m\sim_H m'$ then the long horizontal side
 of $P_m(\pm\overline{x}^m,\overline{y}^m)$ is identified with the long horizontal side of $P_{m'}(\pm\overline{x}^{m'},-\overline{y}^{m'})$;
\item[$(v)$] if $m\sim_v m'$ then the short vertical side
 of $P_m(\overline{x}^m,\pm\overline{y}^m)$ is identified with the short vertical side of $P_{m'}(-\overline{x}^{m'},\pm\overline{y}^{m'})$;
\item[$(vi)$] if $m\sim_h m'$ then the short horizontal side of $P_m(\pm\overline{x}^m,\overline{y}^m)$ is identified with the short horizontal side of $P_{m'}(\pm\overline{x}^{m'},-\overline{y}^{m'})$.
\end{itemize}
\end{remark}

Let $\mathcal A:=\mathcal I\times\Gamma$ and $P_\alpha:=P_m(\gamma(\overline{x}^m,\overline{y}^m))$ if $\alpha=(m,\gamma)\in\mathcal A$.

\begin{remark}\label{rem:corners}
Denote by $\Sigma\subset M(\mathbf P)$ the set of singular points. Singular points arise from some corners of $P_\alpha$, $\alpha\in\mathcal A$.
More precisely,  for every $\alpha=(m,\gamma)\in\mathcal{A}$
\begin{itemize}
 \item[$(i)$] the corner $V^\alpha_{i,i+1}:=\gamma(x^m_i,y^m_{i+1})\in P_\alpha$ for $1\leq i< k_m$  is a singular point with the total angle $6\pi$;
 \item[$(ii)$] the corner $V^\alpha_{i,i}:=\gamma(x^m_i,y^m_{i})\in P_\alpha$ for $1\leq i\leq k_m$ is a regular point.
\end{itemize}
 The other corners $\gamma(0,0)$, $\gamma(x^m_{k_m},0)$, $\gamma(0,y^m_1)$ in $P_\alpha$ can be singular or regular points. More precisely,
\begin{itemize}
 \item[$(iii)$] if $m\sim_V m'\sim_H m''\sim_V m'''\sim_H m$ for some $m', m'',m'''\in\mathcal I$ then $V^\alpha_{0,0}:=\gamma(0,0)\in P_\alpha$
 is a regular point, otherwise it is singular;
 \item[$(iv)$] if  $m\sim_V m'\sim_h m''\sim_V m'''\sim_h m$  for some $m', m'',m'''\in\mathcal I$ then $V^\alpha_{k_m,0}:=\gamma(x^m_{k_m},0)\in P_\alpha$
  is a regular point, otherwise it is singular;
 \item[$(v)$] if $m\sim_v m'\sim_H m''\sim_v m'''\sim_H m$   for some $m', m'',m'''\in\mathcal I$ then $V^\alpha_{0,1}:=\gamma(0,y^m_1)\in P_\alpha$
 is a regular point, otherwise it is singular.
\end{itemize}
\end{remark}

For every $\alpha\in\mathcal A$ there is an open set $U_\alpha\subset M(\mathbf P)\setminus \Sigma$ such that $P_\alpha\setminus \Sigma\subset U_\alpha$ and a chart $\zeta_\alpha:U_\alpha\to \C$ of the translation atlas
of $M(\mathbf P)$
such that its continuous extension $\bar\zeta_\alpha$ is equal to the identity on $P_\alpha$.

Suppose that $e$ is a common side (without ends) of two  polygons $P_\alpha$ and $P_\beta$. As $e\subset U_\alpha\cap U_\beta$, there exists a connected component $C_{\alpha,\beta}^e$ of $U_\alpha\cap U_\beta$ containing $e$.


The following result describes all triples $(\alpha,\beta,C)$ in $D_{\pi/4}=D_{\pi/4}(M(\mathbf P), \mathcal P_{\mathbf P})$ (see Remark~\ref{DBEtheta}) and the corresponding translation vectors $v_{\alpha,\beta}^C$.

\begin{lemma}\label{lem:setD}
Every triple $(\alpha,\beta,C)\in D_{\pi/4}$ is of the form $(\alpha,\beta,C_{\alpha,\beta}^e)$, where  $e$ a common side of polygons $P_\alpha$ and $P_\beta$.
If $C=C_{\alpha,\beta}^e$ and $e$ is:
\begin{itemize}
\item[$(i_h)$] the horizontal side joining $V_{j-1,j}^\alpha$ and $V_{j,j}^\alpha$ for $1< j\leq k_m$ with $\alpha=(m,\gamma)$, then $\beta=(m,\gamma_h\circ\gamma)$ and $v_{\alpha,\beta}^{C}=2y_j^m i$;
\item[$(ii_h)$] the horizontal side joining $V_{0,1}^\alpha$ and $V_{1,1}^\alpha$ (the short horizontal side of $P_\alpha$), $\alpha=(m,\gamma)$ and $m\sim_h m'$, then $\beta=(m',\gamma_h\circ\gamma)$ and $v_{\alpha,\beta}^{C}=(y_1^m+y_1^{m'})i$;
\item[$(iii_h)$] the long horizontal side $P_\alpha$, $\alpha=(m,\gamma)$ and $m\sim_H m'$, then $\beta=(m',\gamma_h\circ\gamma)$ and $v_{\alpha,\beta}^{C}=0$;
\item[$(i_v)$] the vertical side joining $V_{j,j}^\alpha$ and $V_{j,j+1}^\alpha$ for $1\leq j< k_m$ with $\alpha=(m,\gamma)$, then $\beta=(m,\gamma_v\circ\gamma)$ and $v_{\alpha,\beta}^{C}=2x_j^m$;
\item[$(ii_v)$] the vertical side joining $V_{k_m,k_m}^\alpha$ and $V_{k_m,0}^\alpha$ (the short vertical side of $P_\alpha$), $\alpha=(m,\gamma)$ and $m\sim_v m'$, then $\beta=(m',\gamma_v\circ\gamma)$ and $v_{\alpha,\beta}^{C}=x_{k_m}^m+x_{k_{m'}}^{m'}$;
\item[$(iii_v)$] the long vertical side $P_\alpha$, $\alpha=(m,\gamma)$ and $m\sim_V m'$, then $\beta=(m',\gamma_V\circ\gamma)$ and $v_{\alpha,\beta}^{C}=0$.
\end{itemize}
\end{lemma}

\begin{proof}
{
By definition, a triple $(\alpha,\beta,C)$ belongs to $D_{\pi/4}=D_{\pi/4}(M(\mathbf P), \mathcal P_{\mathbf P})$
if there is an orbit segment of the flow $(\varphi^{\pi/4}_t)_{t\in\R}$ passing from  $P_\beta$ to $P_\alpha$
across their common   side $e$ and $e\subset C$. Then $v_{\alpha,\beta}^C$ computes the displacement between local
coordinates in both polygons. Since all sides of the partition $\mathcal P_{\mathbf P}$ are only horizontal and vertical,
we will deal only with horizontal. In the vertical case, the reasoning is analogous.

There are three types of horizontal sides: long horizontal sides, short horizontal sides and sides joining
$V_{j-1,j}^\alpha$ and $V_{j,j}^\alpha$ for $1< j\leq k_m$.

If $e$ is a long horizontal side, then all  points in $e$ have the same local coordinates in both polygons.
It follows that $v_{\alpha,\beta}^C=0$, which confirms $(iii_h)$.

Suppose that $e$  is a side joining  $V_{j-1,j}^\alpha$ and $V_{j,j}^\alpha$ for $1< j\leq k_m$
and an orbit segment of the flow $(\varphi^{\pi/4}_t)_{t\in\R}$ passes from  $P_\beta$ to $P_\alpha$
across $e$. By the gluing rule $(ii)$, we have $P_\beta= P_m(\pm\overline{x}^m,\overline{y}^m)$ and $P_\alpha=P_m(\pm\overline{x}^m,-\overline{y}^m)$.
Moreover, the local coordinate in $P_\beta$ of any point $x \in e$ is of the form $t+iy_j^m$, whereas its local coordinate in $P_\alpha$ is $t-iy_j^m$.
It follows that $v_{\alpha,\beta}^C=\xi_\beta(x)-\xi_\alpha(x)=2y_j^mi$, which confirms $(i_h)$.

If $e$ is a short horizontal side, the arguments are similar. By the gluing rule $(v)$, we have $P_\beta= P_m(\pm\overline{x}^m,\overline{y}^m)$ and $P_\alpha=P_{m'}(\pm\overline{x}^{m'},-\overline{y}^{m'})$ with $m\sim_v m'$. Then for every $x\in e$ we have
\[v_{\alpha,\beta}^C=\xi_\beta(x)-\xi_\alpha(x)=t+iy_j^m-(t-iy_j^{m'})=(y_1^m+y_1^{m'})i,\]
which confirms $(ii_h)$.
}
\end{proof}

The following result describes all pairs $(\alpha,\sigma)$ in $B_{\pi/4}=B_{\pi/4}(M(\mathbf P), \mathcal P_{\mathbf P})$ and $(\beta,\sigma)$ in $E_{\pi/4}=E_{\pi/4}(M(\mathbf P), \mathcal P_{\mathbf P})$ (see Remark~\ref{DBEtheta}) and the corresponding vectors $-\bar\zeta_\alpha(\sigma)$ and $\bar\zeta_\beta(\sigma)$ respectively.

\begin{lemma}\label{lem:setsBE}
Suppose that $(\alpha,\sigma)\in B_{\pi/4}$.
If the singular point $\sigma\in \Sigma$ is of the form:
\begin{itemize}
\item[$(i_B)$] $V_{j,j+1}^\alpha$ for $1\leq j< k_m$ with $\alpha=(m,\gamma)$, then $\gamma=\gamma_h$ or $\gamma_v$ or $\gamma_v\circ\gamma_h$ and $-\bar\zeta_\alpha(\sigma)=-x^m_j+y^m_{j+1}i$ or $x^m_j-y^m_{j+1}i$ or $x^m_j+y^m_{j+1}i$ respectively;
\item[$(ii_B)$] $V_{0,1}^\alpha$ with $\alpha=(m,\gamma)$, then $\gamma=\gamma_h$ and $-\bar\zeta_\alpha(\sigma)=y^m_{1}i$;
\item[$(iii_B)$] $V_{k_m,0}^\alpha$ with $\alpha=(m,\gamma)$, then $\gamma=\gamma_v$ and $-\bar\zeta_\alpha(\sigma)=x^m_{k_m}$;
\item[$(iv_B)$] $V_{0,0}^\alpha$ with $\alpha=(m,\gamma)$, then $\gamma=id$ and $-\bar\zeta_\alpha(\sigma)=0$.
\end{itemize}
Suppose that $(\beta,\sigma)\in E_{\pi/4}$.
If the singular point $\sigma\in \Sigma$ is of the form:
\begin{itemize}
\item[$(i_E)$] $V_{j,j+1}^\beta$ for $1\leq j< k_m$ with $\beta=(m,\gamma)$, then $\gamma=id$ or $\gamma_h$ or $\gamma_v$  and $\bar{\zeta}_\beta(\sigma)=x^m_j+y^m_{j+1}i$ or $x^m_j-y^m_{j+1}i$ or $-x^m_j+y^m_{j+1}i$ respectively;
\item[$(ii_E)$] $V_{0,1}^\beta$ with $\beta=(m,\gamma)$, then $\gamma=\gamma_v$ and $\bar\zeta_\beta(\sigma)=y^m_{1}i$;
\item[$(iii_E)$] $V_{k_m,0}^\beta$ with $\beta=(m,\gamma)$, then $\gamma=\gamma_h$ and $\bar\zeta_\beta(\sigma)=x^m_{k_m}$;
\item[$(iv_E)$] $V_{0,0}^\beta$ with $\beta=(m,\gamma)$, then $\gamma=\gamma_v\circ\gamma_h$ and $\bar\zeta_\beta(\sigma)=0$.
\end{itemize}
\end{lemma}

\begin{proof}
{
Since the descriptions of the sets $B_{\pi/4}$ and $E_{\pi/4}$ result from similar reasoning, we will focus only on $B_{\pi/4}$.

By definition,  $(\alpha,\sigma)\in\mathcal{A}\times\Sigma$ belongs to $B_{\pi/4}$
if there is an orbit segment of the flow $(\varphi^{\pi/4}_t)_{t\in\R}$ in $P_\alpha$ starting from the singular point $\sigma\in P_\alpha$.
By Remark~\ref{rem:corners}, there are four types of singularities in $M(\mathbf P)$: $V_{0,0}^\alpha$, $V_{0,1}^\alpha$, $V_{k_m,0}^\alpha$
or  $V_{j,j+1}^\alpha$ for $1\leq j< k_m$.

Suppose that  $\sigma = V_{j,j+1}^\alpha$ for some $1\leq j< k_m$, where $\alpha=(m,\gamma)$. Then $P_\alpha$ is of the form
$P_m(\overline{x}^m,-\overline{y}^m)$ or $P_m(-\overline{x}^m,-\overline{y}^m)$ or $P_m(-\overline{x}^m,-\overline{y}^m)$ or $P_m(\overline{x}^m,\overline{y}^m)$. However, the only corner in $P_m(\overline{x}^m,\overline{y}^m)$ from which an orbit segment in direction $\pi/4$ comes out is $(0,0)$,
so the case $P_\alpha=P_m(\overline{x}^m,\overline{y}^m)$ cannot occur.
In other cases, local coordinates $\bar\zeta_\alpha(\sigma)$ are of the form  $x^m_j-y^m_{j+1}i$ or $-x^m_j+y^m_{j+1}i$ or $-x^m_j-y^m_{j+1}i$ respectively, which confirms $(i_B)$.

Remaining types of singularities (i.e.\ $V_{0,0}^\alpha$, $V_{0,1}^\alpha$, $V_{k_m,0}^\alpha$) follow by the same arguments
and we leave it to the reader.
}\end{proof}

\section{Smooth curves of  billiard tables}
Suppose that $J\ni s\mapsto \mathbf{P}(s)\in\mathscr{P}$ ($J\subset\R$ is an open interval) is a $C^\infty$-curve of polygonal tables.
Assume that all generalized polygons $\mathbf{P}(s)$, $s\in J$ have the same combinatorial data
$\{\mathcal I, (\gamma_m)_{m\in\mathcal I}, (k_m)_{m\in\mathcal I}, \sim_H,\sim_V,\sim_h,\sim_v\}$.
Then for every $ m\in \mathcal I$ there exists a
 $C^\infty$-map
\[J\ni s\mapsto (\overline{x}^m(s),\overline{y}^m(s))\in\Xi\]
so that for every $s\in J$ we have
\[\mathbf{P}(s)= \biguplus_{m\in\mathcal I}\gamma_m P(\overline{x}^m(s),\overline{y}^m(s))\]
with the gluing rules of basic polygons given by the four binary relations $\sim_V$, $\sim_H$, $\sim_v$, $\sim_h$.

The smooth curve of polygons $s\mapsto \mathbf{P}(s)$ provides a $C^\infty$ curve $s\mapsto M(\mathbf{P}(s))$ in the moduli space of translation surfaces $\mathcal M$.
Then for every $s\in J$ the surface $M(\mathbf{P}(s))\in \mathcal M$ has a natural partition into basic polygons
\begin{equation*}
\big\{P_m(\gamma(\overline{x}^m(s),\overline{y}^m(s))):(m,\gamma)\in \mathcal I\times\Gamma\big\}
\end{equation*}
so that their sides are identified according to the rules described in Remark~\ref{rem:gluing_rules}.

Let us consider two finite subsets in $C^\infty(J,\R_{>0})$ given by
\[\mathscr{X}_{\mathbf P}:=\{ x^m_j:m\in\mathcal I, 1\leq j\leq k_m\},\quad
\mathscr{Y}_{\mathbf P}:=\{ y^m_j:m\in\mathcal I, 1\leq j\leq k_m\}.\]

For any sequence  $(g_k)_{k=1}^n$ of maps in $C^\infty(J,\R)$ denote by $|W|((g_k)_{k=1}^n)$ the absolute value of its Wronskian, i.e.\
\[|W|((g_k)_{k=1}^n)(s)=\Big|\det\Big[\frac{d^{j-1}}{ds^{j-1}}g_k(s)\Big]_{j,k=1,\ldots,n}\Big|.\]

The following lemma is a straightforward consequence of Lebesgue's density theorem.
\begin{lemma}\label{lem:wron}
Suppose that $(g_k)_{k=1}^n$ is a sequence of maps in $C^\infty(J,\R)$ such that
\[|W|((g_k)_{k=1}^n)(s)>0\quad \text{for a.e.}\quad s\in J.\]
Then for a.e.\ $s\in J$ and for every sequence $(m_k)_{k=1}^n$ of at least one non-zero integer numbers we have $\sum_{k=1}^nm_kg_k(s)\neq 0$.
\end{lemma}

Since the absolute value of Wronskian does not depend on the order of the sequence, we can also define  the absolute value of Wronskian for
finite subsets in $C^\infty(J,\R)$ letting
\[|W|\{g_k:1\leq k\leq n\}:=|W|((g_k)_{k=1}^n),\]
if $g_k$, $1\leq k\leq n$ are distinct maps.

\begin{theorem}\label{thm:mainsurf}
Let $\ell:J\to\R_{>0}$ be a $C^\infty$-map. Suppose that
\begin{itemize}
\item[$(i)$] $|W|(\mathscr{X}_{\mathbf P}\cup\mathscr{Y}_{\mathbf P}\cup\{\ell\})(s)> 0$ for a.e.\ $s\in J$, and
\item[$(ii_{+-})$] $[\mathbf{x},\ell](s)\geq 0$ and $[\mathbf{y},\ell](s)<0$ for all $\mathbf{x}\in\mathscr{X}_{\mathbf P}$, $\mathbf{y}\in\mathscr{Y}_{\mathbf P}$ and a.e.\ $s\in J$, or
\item[$(ii_{-+})$] $[\mathbf{x},\ell](s)\leq 0$ and $[\mathbf{y},\ell](s)>0$ for all $\mathbf{x}\in\mathscr{X}_{\mathbf P}$, $\mathbf{y}\in\mathscr{Y}_{\mathbf P}$ and a.e.\ $s\in J$.
\end{itemize}
 Then  for a.e.\ $s\in J$ the translation flow $(\varphi^{\pi/4}_t)_{t\in\R}$  on
$M(\mathbf{P}(s))$ is uniquely ergodic.
\end{theorem}

\begin{proof}
Since the translation flow $(\varphi^{\pi/4}_t)_{t\in\R}$  on
$M(\mathbf{P}(s))$ coincide  with the vertical flow on $r_{\pi/4}M(\mathbf{P}(s))$ { (see Remark~\ref{rem:rot})} and the vertical direction does not belong to $\operatorname{dir} r_{\pi/4}\mathcal P_{\mathbf P(s)}$ for all $s\in J$, we can apply Theorem~\ref{thm:gencrit}
{(together with Proposition~\ref{prop:mas})} to the $C^\infty$-map $J\ni s\mapsto r_{\pi/4}M(\mathbf P(s))\in \mathcal{M}$ and to the smooth family of polygonal partitions $r_{\pi/4}\mathcal P_{\mathbf P(s)}$ of the translation surface $r_{\pi/4}M(\mathbf P(s))$ for $s\in J$. We need to verify the conditions $(i)$ and $(ii_\pm)$ from Theorem~\ref{thm:gencrit}.

Using Remark~\ref{DBEtheta} and Lemmas~\ref{lem:setD}-\ref{lem:setsBE}, we can easily localize the corresponding finite subset $\mathscr D$, $\mathscr B$, $\mathscr E$ in $C^{\infty}(J,\C)$. More precisely,
\begin{align}
\label{eq:D}
\mathscr D&\subset e^{i\pi/4}\big(\{0\}\cup(\mathscr{X}_{\mathbf P}+\mathscr{X}_{\mathbf P})\cup i(\mathscr{Y}_{\mathbf P}+\mathscr{Y}_{\mathbf P})\big)\\
\label{eq:BE}
\mathscr B, \mathscr E&\subset e^{i\pi/4}\big((\mathscr{X}_{\mathbf P}-i\mathscr{Y}_{\mathbf P})\cup(-\mathscr{X}_{\mathbf P}+i\mathscr{Y}_{\mathbf P})\cup\big((\{0\}\cup\mathscr{X}_{\mathbf P})+i(\{0\}\cup\mathscr{Y}_{\mathbf P})\big)\big).
\end{align}

Suppose that $f\in \mathscr B$, $g\in \mathscr E$ and $(n_h)_{h\in \mathscr D}$ is a sequence of numbers in $\Z_{\geq 0}$ such that the map $f+g+\sum_{h\in \mathscr D}n_hh$ is non-zero.
By \eqref{eq:D} and \eqref{eq:BE}, there exist integer numbers $a_{\mathbf x}$ for $\mathbf x\in \mathscr{X}_{\mathbf P}$ and $b_{\mathbf y}$ for $\mathbf y\in \mathscr{Y}_{\mathbf P}$ such that
\[f+g+\sum_{h\in \mathscr D}n_hh=e^{i\pi/4}\Big(\sum_{\mathbf x\in \mathscr{X}_{\mathbf P}}a_{\mathbf x}\mathbf x+i\sum_{\mathbf y\in \mathscr{Y}_{\mathbf P}}b_{\mathbf y}\mathbf y\Big).\]
As the left hand side is a non-zero map, at least one integer number $a_{\mathbf x}$ or $b_{\mathbf y}$  is non-zero. By the assumption $(i)$ and Lemma~\ref{lem:wron}, for a.e.\ $s\in J$ we have
\begin{align*}
\rp f(s)+\rp g(s)+\sum_{h\in \mathscr D}n_h\rp h(s)=
\frac{1}{\sqrt{2}}\Big(\sum_{\mathbf x\in \mathscr{X}_{\mathbf P}}a_{\mathbf x}\mathbf x(s)-\sum_{\mathbf y\in \mathscr{Y}_{\mathbf P}}b_{\mathbf y}\mathbf y(s)\Big)\neq 0.
\end{align*}
Therefore the condition $(i)$ from Theorem~\ref{thm:gencrit} holds.

\smallskip
In order to verify the conditions $(ii_\pm)$ in Theorem~\ref{thm:gencrit} we take any non-zero map $h\in \mathscr D$. In view of \eqref{eq:D},
\begin{align}
\label{typex}
h=e^{i\pi/4}(\mathbf x_1+\mathbf x_2)\quad&\text{for some}\quad  \mathbf x_1,\mathbf x_2\in \mathscr{X}_{\mathbf P}, or\\
\label{typey}
h=e^{i\pi/4}i(\mathbf y_1+\mathbf y_2)\quad&\text{for some}\quad  \mathbf y_1,\mathbf y_2\in \mathscr{Y}_{\mathbf P}.
\end{align}
By Lemma~\ref{lem:setD}, there are   maps in $\mathscr D$ of both types \eqref{typex} and \eqref{typey}.
Moreover,
\begin{align*}
[\rp h(s),\ell(s)]&=\Big[\frac{\mathbf x_1+\mathbf x_2}{\sqrt{2}},\ell\Big](s)=\frac{[\mathbf{x}_1,\ell](s)+[\mathbf{x}_2,\ell](s)}{\sqrt{2}},\text{ or}\\
[\rp h(s),\ell(s)]&=\Big[-\frac{\mathbf y_1+\mathbf y_2}{\sqrt{2}},\ell\Big](s)=-\frac{[\mathbf{y}_1,\ell](s)+[\mathbf{y}_2,\ell](s)}{\sqrt{2}}.
\end{align*}

Under the assumption $(ii_{+-})$, it follows that for all  $h\in \mathscr D$ and  $s\in J$ we have $[\rp h(s),\ell(s)]\geq 0$, and if $h$ is of type \eqref{typey}
then $[\rp h(s),\ell(s)]> 0$ for a.e.\ $s\in J$. This gives the condition $(ii_+)$ in Theorem~\ref{thm:gencrit}.

Under the assumption $(ii_{-+})$, we obtain that $[\rp h(s),\ell(s)]\leq 0$ for all $h\in \mathscr D$ and $s\in J$, and if $h$ is of type \eqref{typey}
then $[\rp h(s),\ell(s)]< 0$ for a.e.\ $s\in J$. This gives the condition $(ii_-)$ in Theorem~\ref{thm:gencrit}.

Therefore, by Theorem~\ref{thm:gencrit}, $r_{\pi/4}M(\mathbf{P}(s))$ is recurrent for a.e.\ $s\in J$. {
In view of Proposition~\ref{prop:mas} and Remark~\ref{rem:rot},} it follows that the translation flow $(\varphi^{\pi/4}_t)_{t\in\R}$  on
$M(\mathbf{P}(s))$ is uniquely ergodic for a.e.\ $s\in J$.
\end{proof}

\section{Unique ergodicity of the billiard flow  on $\mathcal D$ restricted to $\mathcal S_s$.}\label{sec:uniqerg}
Let us consider the billiard flow on a table
\[\mathcal D=\mathcal{D}^{(\overline{\alpha}^{++},\overline{\beta}^{++})(\overline{\alpha}^{+-},\overline{\beta}^{+-})}_{ (\overline{\alpha}^{-+},\overline{\beta}^{-+})(\overline{\alpha}^{--},\overline{\beta}^{--})}.\]
Without loss of generality we can assume  that $\beta^t\leq\beta^b\leq\beta^l\leq\beta^r$.
Recall the phase space $S^1\mathcal D$ of the billiard flow on $\mathcal{D}$ splits into invariant subsets $\mathcal S_s$, $s\in(\beta^t,a)$. By Proposition~\ref{prop:bilflow}, if $s\neq b$ then the billiard flow restricted to $\mathcal S_s$
is topologically conjugate to the directional billiard flow  in  directions $\pm\pi/4,\pm 3\pi/4$ on $\sigma_s( S_s)\in\mathscr{P}$. For more precise description of every generalized polygon $\sigma_s(S_s)$ we need to consider the partition $\mathscr{J}$ of the interval $(\beta^t,a)$ into open intervals by the points $\alpha^{\pm\pm}_i$, $\beta^{\pm\pm}_i$ for $1\leq i\leq k(\overline{\alpha}^{\pm\pm}, \overline{\beta}^{\pm\pm})$. For every open interval $J\in \mathscr{J}$
the generalized polygons $\sigma_s(S_s)$, $s\in J$ have the same combinatorial data and the map $J\ni s\mapsto \sigma_s(S_s)\in \mathscr{P}$ is of class $C^\infty$.
For every $J\in \mathscr{J}$ denote by $l=l^{\pm\pm}_J$ the  largest integer number between $0$ and $k(\overline{\alpha}^{\pm\pm}, \overline{\beta}^{\pm\pm})$ so that
$J\subset (\beta^{\pm\pm}_{l},\alpha^{\pm\pm}_{l-1})$.

{ The following precise description of $\sigma_s(S_s)$ follows directly from the shape of the set $\mathcal D$ and the definition of $\sigma_s$.}
\begin{proposition}\label{prop:desc}
For every $J\in\mathscr J$ all generalized polygons $\sigma_s(S_s)$, $s\in J$ belong to the family $\mathscr P$ (introduced in Section~\ref{sec:poly}) and are described as follows:
\begin{itemize}
\item[$(i)$] if $J\subset (\beta^t,\beta^b)$ then  $\sigma_s(S_s)$ for $s\in J$ consists of two  basic polygons
\[P_{++}(\overline{x}^{++}(s),\overline{y}^{++}(s)), \quad P_{-+}(-\overline{x}^{-+}(s),\overline{y}^{-+}(s))\]
with $k(\overline{x}^{\pm+}(s),\overline{y}^{\pm+}(s))=l^{\pm+}_J$ glued according to the rule $++\sim_V-+$  and
    \[
   x_i^{\pm+}(s)=\int_{\alpha^{\pm+}_i}^ae(\lambda,s)\,d\lambda,\quad y_i^{\pm+}(s)=\int_{\beta^{\pm+}_i}^se(\lambda,s)\,d\lambda=\ell(s)-\int_{-\infty}^{\beta^{\pm+}_i}e(\lambda,s)\,d\lambda
    \]
    for $1\leq i\leq l^{\pm+}_J$;
\item[$(ii)$] if $J\subset (\beta^b,b)$ then  $\sigma_s(S_s)$ for $s\in J$ consists of four  basic polygons
    \begin{align*}
    &P_{++}(\overline{x}^{++}(s),\overline{y}^{++}(s)),\ P_{-+}(-\overline{x}^{-+}(s),\overline{y}^{-+}(s)),\\
    &P_{+-}(-\overline{x}^{+-}(s),\overline{y}^{+-}(s)),\ P_{--}(\overline{x}^{--}(s),\overline{y}^{--}(s))
    \end{align*}
    glued according to the rules
\begin{align*}
  ++\sim_V-+, +-\sim_V-- & \ \text{ if }\  J\subset (\beta^b,\beta^l)\\
  ++\sim_V-+,\ -+\sim_v--,\ --\sim_V+- & \ \text{ if }\ J\subset (\beta^l,\beta^r) \\
  ++\sim_V-+,\ -+\sim_v--,\ --\sim_V+-,\ +-\sim_v++ & \ \text{ if }\ J\subset (\beta^r,b),
\end{align*}
with $k(\overline{x}^{\pm\pm}(s),\overline{y}^{\pm\pm}(s))=l^{\pm\pm}_J$ and
\[
   x_i^{\pm\pm}(s)=\int_{\alpha^{\pm\pm}_i}^ae(\lambda,s)\,d\lambda,\quad y_i^{\pm\pm}(s)=\int_{\beta^{\pm\pm}_i}^se(\lambda,s)\,d\lambda=\ell(s)-\int_{-\infty}^{\beta^{\pm\pm}_i}e(\lambda,s)\,d\lambda
\]
    for $1\leq i\leq l^{\pm\pm}_J$. If $J\subset (\beta^b,\beta^l)$ then $\sigma_s(S_s)$ is not connected and it is the union of two polygons: $\sigma_s(S^+_s)$ glued from $P_{++}$ and $P_{-+}$, and $\sigma_s(S^-_s)$ glued from $P_{+-}$ and $P_{--}$;
\item[$(iii)$] if $J\subset (b,a)$ then  $\sigma_s(S_s)$ for $s\in J$ consists of four  basic polygons
    \begin{align*}
    &P_{++}(\overline{x}^{++}(s),\overline{y}^{++}(s)),\ P_{-+}(-\overline{x}^{-+}(s),\overline{y}^{-+}(s)),\\
    &P_{+-}(\overline{x}^{+-}(s),-\overline{y}^{+-}(s)),\ P_{--}(-\overline{x}^{--}(s),-\overline{y}^{--}(s))
    \end{align*}
    glued according to the rules
    \[++\sim_V-+,\ -+\sim_H--,\ --\sim_V+-,\ +-\sim_H++\]
     with $k(\overline{x}^{\pm\pm}(s),\overline{y}^{\pm\pm}(s))=l^{\pm\pm}_J$ and
    \begin{align*}
   &x_i^{\pm\pm}(s)=\int_{\alpha^{\pm\pm}_i}^ae(\lambda,s)\,d\lambda\ \text{ for }\ 1\leq i< l^{\pm\pm}_J,\\ &x_{l^{\pm\pm}_J}^{\pm\pm}(s)=\int_{s}^ae(\lambda,s)\,d\lambda=\ell(s)=\int_{-\infty}^be(\lambda,s)\,d\lambda,\\
   &y_i^{\pm\pm}(s)=\int_{\beta^{\pm\pm}_i}^b e(\lambda,s)\,d\lambda\ \text{ for }\ 1\leq i\leq l^{\pm\pm}_J.
    \end{align*}
\end{itemize}
\end{proposition}
\begin{figure}[h]
\includegraphics[width=1 \textwidth]{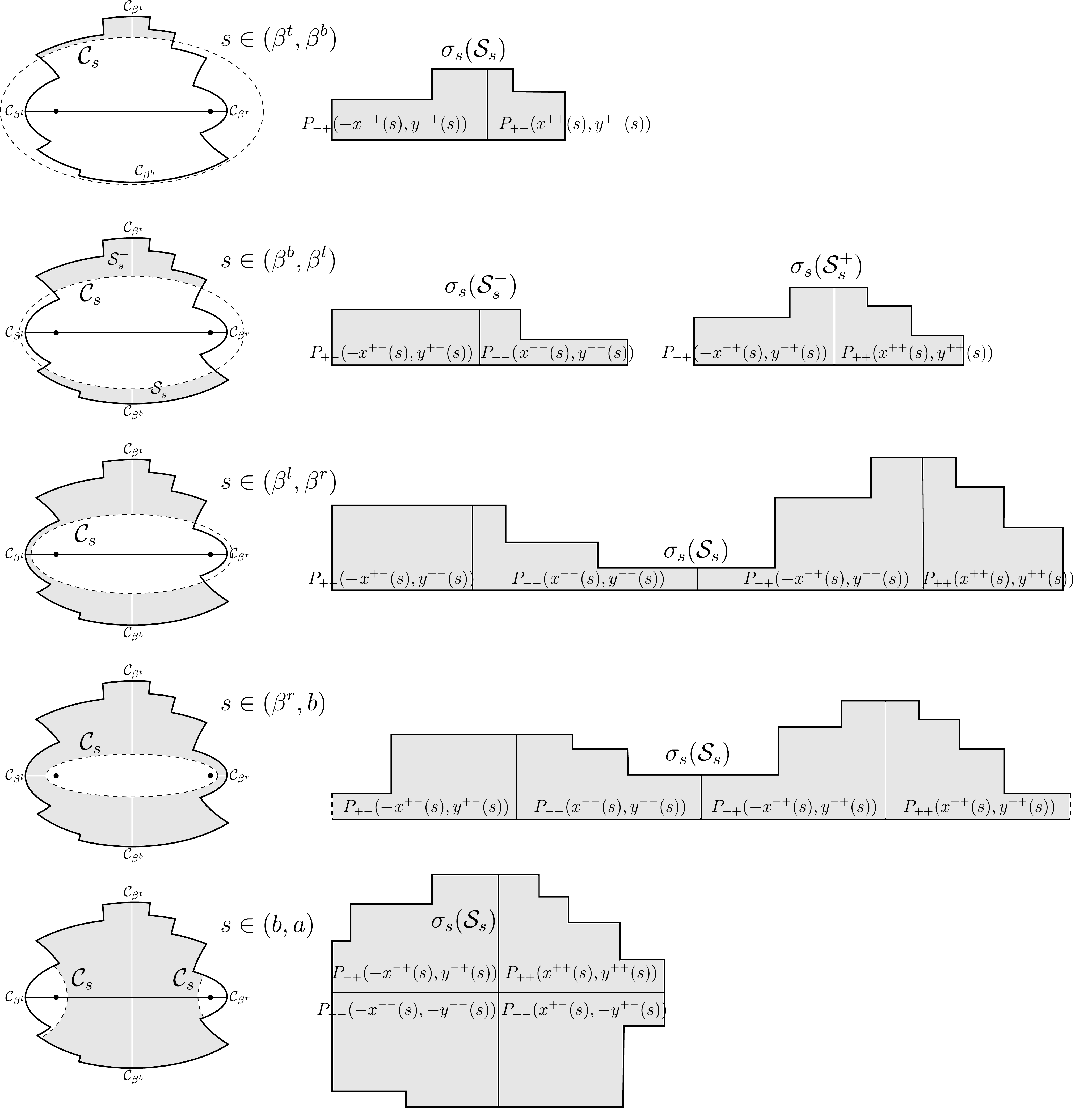}
\caption{All possible types of the polygon $\sigma_s(S_s)$.}
\label{fig:sigmaSs}
\end{figure}

For every $s<a$ denote by $\Delta_s\subset\R$ the domain of $\lambda\mapsto e(\lambda,s)=\tfrac{1}{\sqrt{(a-\lambda)(b-\lambda)(s-\lambda)}}$.
If $s<b$ then $\Delta_s=(-\infty,s)\cup(b,a)$ and
\[\int_{\Delta_s} e(\lambda,s)\,d\lambda= 2\int_b^a e(\lambda,s)\,d\lambda<+\infty.\]
If $b<s<a$ then $\Delta_s=(-\infty,b)\cup(s,a)$ and
\[\int_{\Delta_s} e(\lambda,s)\,d\lambda= 2\int_{-\infty}^b e(\lambda,s)\,d\lambda<+\infty.\]

Take an interval $J\in\mathscr{J}$ and an open interval $D$ such that $D\subset \Delta_s$ for every $s\in J$.
Denote by $\xi_D:J\to\R_{>0}$ the map given by $\xi_D(s)=\int_De(\lambda,s)\,d\lambda$. Then $\xi_D$ is a $C^\infty$-map
such that
\begin{equation}\label{eq:derxi}
\frac{d^k}{ds^k}\xi_D(s)=\frac{(2k-1)!!}{2^k}\int_D\frac{e(\lambda,s)}{(\lambda-s)^k}\,d\lambda\ \text{ for all }\ k\geq 1, \ s\in J.
\end{equation}

Fix an interval $J\in\mathscr{J}$. Then the family of polygons $\sigma_s(S_s)$, $s\in J$ is determined
be a $C^\infty$-map $J\ni s\mapsto \mathbf{P}(s)\in\mathscr{P}$. { In view of Proposition~\ref{prop:desc},
we have the following result.}

\begin{corollary}\label{cor:XYl}
For every interval $J\in\mathscr J$ there exist
\[a> \alpha_{1}>\alpha_{2}>\ldots>\alpha_m>b>\beta_n>\ldots>\beta_2>\beta_1\geq 0 \ \text{ with }\ J\subset(\beta_n,\alpha_m)\]
such that  if  $J\subset(\beta_t,b)$ then
\[
 \mathscr{X}_{\mathbf{P}}\cup\{\ell\}=\{\ell=\xi_{(b,a)},\ \xi_{(\alpha_i,a)}: 1\leq i\leq m\},\quad
   \mathscr{Y}_{\mathbf{P}}=\{\ell-\xi_{(-\infty,\beta_j)}: 1\leq j\leq n\}
 \]
and if  $J\subset(b,a)$ then
\[
\mathscr{X}_{\mathbf{P}}\cup\{\ell\}=\{\ell=\xi_{(-\infty,b)},\ \xi_{(\alpha_i,a)}: 1\leq i\leq m\},\quad
   \mathscr{Y}_{\mathbf{P}}=\{\xi_{(\beta_j,b)}:1\leq j\leq n\}.
\]
\end{corollary}

\begin{lemma}[Lemma~3.3 in \cite{Fr-Shi-Ul}]\label{lem:nonzerodet}
Let $f:(c_1,c_2)\cup(c_3,c_4)\to\R_{>0}$ $(-\infty\le c_1<c_2<c_3<c_4\le \infty)$ be a
positive  continuous function with finite integrals
$
\int_{c_1}^{c_2}f(\lambda)\, d \lambda$
and
 $ \int_{c_3}^{c_4}f(\lambda)\, d \lambda$.
If $\{A_i:1\leq i\leq
k\}$ is a family of pairwise disjoint subintervals of
$(c_1,c_2)\cup(c_3,c_4)$, then we have
\begin{equation*}\label{ineq:int}
\det\left[\int_{A_{i}}\frac{f(\lambda)\,d\lambda}{(\lambda-s)^{j-1}}\right]_{i,j=1,\ldots,k}\neq 0\quad\text{for every }\ s\in(c_2,c_3).
\end{equation*}
\end{lemma}
As a consequence, in view of \eqref{eq:derxi}, we obtain the following result.
\begin{corollary}\label{cor:wron}
Suppose that $D_i$, $1\leq i\leq k$ is a family of pairwise disjoint open intervals such that $\bigcup_{i=1}^kD_i\subset\Delta_s$ for all $s\in J$.
Then
\[|W|((\xi_{D_j})_{j=1}^k)(s)> 0\quad\text{for all}\quad s\in J.\]
\end{corollary}

\begin{lemma}\label{lem:bracket}
Suppose that $D_1$, $D_2$ are disjoint open intervals such that $D_1\cup D_2\subset\Delta_s$ for all $s\in J$. Then
\begin{align*}
[\xi_{D_1},\xi_{D_2}](s)<0\ \text{ for all }\ s\in J\ &\text{ if }\ D_2<D_1<J\ \text{ or }\ J<D_2<D_1,\\
[\xi_{D_1},\xi_{D_2}](s)>0\ \text{ for all }\ s\in J\ &\text{ if }\ D_2<J<D_1.
\end{align*}
\end{lemma}

\begin{proof}
In view of \eqref{eq:derxi}, we have
\begin{align*}
[\xi_{D_1},\xi_{D_2}](s)&=\frac{1}{2}\int_{D_1\times D_2}\left(\frac{e(\lambda_1,s)e(\lambda_2,s)}{\lambda_1-s}-\frac{e(\lambda_1,s)e(\lambda_2,s)}{\lambda_2-s}\right)d\lambda_1d\lambda_2\\
&=\frac{1}{2}\int_{D_1\times D_2}\frac{e(\lambda_1,s)e(\lambda_2,s)(\lambda_2-\lambda_1)}{(\lambda_1-s)(\lambda_2-s)}d\lambda_1d\lambda_2.
\end{align*}
Since for all $\lambda_1\in D_1$,  $\lambda_2\in D_2$, $s\in J$ we have $\lambda_2-\lambda_1<0$
and $(\lambda_1-s)(\lambda_2-s)>0$ if $ D_2<D_1<J$ or $J<D_2<D_1$ and $(\lambda_1-s)(\lambda_2-s)<0$ if $ D_2<J<D_1$, this implies the required inequalities.
\end{proof}

\begin{theorem}\label{thm:wronbrack}
For every $J\in\mathscr{J}$ we have $|W|(\mathscr X_{\mathbf P}\cup\mathscr Y_{\mathbf P}\cup\{\ell\})(s)> 0$ for every $s\in J$. Moreover,
for all $\mathbf x\in \mathscr X_{\mathbf P}$, $\mathbf y\in \mathscr Y_{\mathbf P}$ and $s\in J$ we have
\begin{align*}
&[\mathbf x,\ell](s)\leq 0\ \text{ and }\ [\mathbf y,\ell](s)>0\ \text{ if }\  J\subset(\beta^t,b),\\
&[\mathbf x,\ell](s)\geq 0\ \text{ and }\ [\mathbf y,\ell](s)<0\ \text{ if }\  J\subset(b,a).
\end{align*}
\end{theorem}

\begin{proof}
The proof relies on Corollary~\ref{cor:wron}, Lemma~\ref{lem:bracket} and the fact the both the Wronskian $W$ and the bracket $[\,\cdot\, ,\,\cdot\, ]$ are alternating multilinear forms.

\textbf{Case $J\subset (\beta^t,b)$.} In view of Corollary~\ref{cor:XYl}, $J\subset(\beta_n,b)$ and for every $s\in J$ we have $\ell=\xi_{(b,a)}$ and
\begin{align*}
|W|&(\mathscr X_{\mathbf P}\cup\mathscr Y_{\mathbf P}\cup\{\ell\})(s)=|W|(\xi_{(\alpha_1,a)},\ldots,\xi_{(\alpha_m,a)},\ell-\xi_{(-\infty,\beta_1)}, \ldots, \ell-\xi_{(-\infty,\beta_n)},\ell)(s)\\
&=|W|(\xi_{(\alpha_1,a)},\ldots,\xi_{(\alpha_m,a)},\xi_{(b,a)},\xi_{(-\infty,\beta_1)}, \ldots, \xi_{(-\infty,\beta_n)})(s)\\
&=|W|(\xi_{(\alpha_1,a)},\xi_{(\alpha_2,\alpha_1)},\ldots,\xi_{(\alpha_m,\alpha_{m-1})},\xi_{(b,\alpha_m)},\xi_{(-\infty,\beta_1)},\xi_{(\beta_1,\beta_2)}, \ldots, \xi_{(\beta_{n-1},\beta_n)})(s).
\end{align*}
Since the intervals
\[(-\infty,\beta_1), (\beta_1,\beta_2), \ldots,(\beta_{n-1},\beta_n), (b,\alpha_m), (\alpha_m,\alpha_{m-1}),\ldots,(\alpha_2,\alpha_1), (\alpha_1,a)\]
are pairwise disjoint, by Corollary~\ref{cor:wron}, we have $|W|(\mathscr X_{\mathbf P}\cup\mathscr Y_{\mathbf P}\cup\{\ell\})(s)> 0$ for every $s\in J$.

By  Corollary~\ref{cor:XYl}, if $\mathbf x\in \mathscr X_{\mathbf P}$ then $\mathbf x=\xi_{(\alpha_i,a)}$ for some $1\leq i\leq m$ or $\mathbf x=\ell=\xi_{(b,a)}$.
Since $J<(b,\alpha_i)<(\alpha_i,a)$, by  Lemma~\ref{lem:bracket}, for every $s\in J$ we have
\[[\xi_{(\alpha_i,a)},\ell](s)=[\xi_{(\alpha_i,a)},\xi_{(b,\alpha_i)}](s)<0.\]
As $[\ell,\ell]=0$, we obtain that $[\mathbf x,\ell](s)\leq 0$ for all $\mathbf x\in \mathscr X_{\mathbf P}$ and $s\in J$.

By  Corollary~\ref{cor:XYl}, if $\mathbf y\in \mathscr Y_{\mathbf P}$ then $\mathbf y=\ell-\xi_{(-\infty,\beta_j)}$ for some $1\leq j\leq n$.
Since $(-\infty,\beta_j)<J<(b,a)$, by  Lemma~\ref{lem:bracket}, for every $s\in J$ we have
\[[\ell-\xi_{(-\infty,\beta_j)},\ell](s)=-[\xi_{(-\infty,\beta_j)},\xi_{(b,a)}](s)=[\xi_{(b,a)},\xi_{(-\infty,\beta_j)}](s)>0.\]
Therefore,  $[\mathbf y,\ell](s)> 0$ for all $\mathbf y\in \mathscr Y_{\mathbf P}$ and $s\in J$.

\textbf{Case $J\subset (b,a)$.} In view of Corollary~\ref{cor:XYl}, $J\subset(b,\alpha_m)$ and for every $s\in J$ we have $\ell=\xi_{(-\infty,b)}$ and
\begin{align*}
|W|&(\mathscr X_{\mathbf P}\cup\mathscr Y_{\mathbf P}\cup\{\ell\})(s)=|W|(\xi_{(\alpha_1,a)},\ldots,\xi_{(\alpha_m,a)},\xi_{(\beta_1,b)}, \ldots, \xi_{(\beta_n,b)},\xi_{(-\infty,b)})(s)\\
&=|W|(\xi_{(\alpha_1,a)},\xi_{(\alpha_2,\alpha_1)},\ldots,\xi_{(\alpha_m,\alpha_{m-1})},\xi_{(-\infty,\beta_1)},\xi_{(\beta_1,\beta_2)}, \ldots, \xi_{(\beta_{n-1},\beta_n)},\xi_{(\beta_n,b)})(s).
\end{align*}
Since the intervals
\[(-\infty,\beta_1), (\beta_1,\beta_2), \ldots,(\beta_{n-1},\beta_n), (\beta_n,b), (\alpha_m,\alpha_{m-1}),\ldots,(\alpha_2,\alpha_1), (\alpha_1,a)\]
are pairwise disjoint, by Corollary~\ref{cor:wron}, we have $|W|(\mathscr X_{\mathbf P}\cup\mathscr Y_{\mathbf P}\cup\{\ell\})(s)> 0$ for every $s\in J$.

By  Corollary~\ref{cor:XYl}, if $\mathbf x\in \mathscr X_{\mathbf P}$ then $\mathbf x=\xi_{(\alpha_i,a)}$ for some $1\leq i\leq m$ or $\mathbf x=\ell=\xi_{(-\infty,b)}$.
Since $(-\infty,b)<J<(\alpha_i,a)$, by  Lemma~\ref{lem:bracket}, for every $s\in J$ we have
\[[\xi_{(\alpha_i,a)},\ell](s)=[\xi_{(\alpha_i,a)},\xi_{(-\infty,b)}](s)>0.\]
As $[\ell,\ell]=0$, we obtain that $[\mathbf x,\ell](s)\geq 0$ for all $\mathbf x\in \mathscr X_{\mathbf P}$ and $s\in J$.

By  Corollary~\ref{cor:XYl}, if $\mathbf y\in \mathscr Y_{\mathbf P}$ then $\mathbf y=\xi_{(\beta_j,b)}$ for some $1\leq j\leq n$.
Since $(-\infty,\beta_j)<(\beta_j,b)<J$, by  Lemma~\ref{lem:bracket}, for every $s\in J$ we have
\[[\xi_{(\beta_j,b)},\ell](s)=[\xi_{(\beta_j,b)},\xi_{(-\infty,\beta_j)}](s)<0.\]
Therefore,  $[\mathbf y,\ell](s)< 0$ for all $\mathbf y\in \mathscr Y_{\mathbf P}$ and $s\in J$.
\end{proof}

\begin{proof}[Proof of Theorem~\ref{thm:main}]
By Proposition~\ref{prop:bilflow}, for every  $s\in(\min\{\beta^t,\beta^b\},b)\cup (b,a)$ the billiard flow on $\mathcal D$ restricted to $\mathcal S_s$
is topologically conjugated to the directional billiard flow in  directions $\pm\pi/4$, $\pm3\pi/4$ on the polygon $\sigma_s(S_s)\in\mathscr{P}$. Moreover,
for $s\in(\min\{\beta^t,\beta^b\},\min\{\beta^l,\beta^r\})$ the polygon $\sigma_s(S_s)$ is the union of two connected polygons
$\sigma_s(S^+_s)$ and $\sigma_s(S^-_s)$. To conclude the proof we need to show that for every open interval $J$
of the partition $\mathscr J$ and for a.e.\ $s\in J$ the flow $(\varphi^{\pi/4}_t)_{t\in\R}$ on $M(\sigma_s(S_s))$ (or $M(\sigma_s(S^\pm_s))$ if $J\subset (\min\{\beta^t,\beta^b\},\min\{\beta^l,\beta^r\})$) is uniquely ergodic.

By Theorem~\ref{thm:wronbrack}, for every $J\in\mathscr J$ the map $J\ni s\mapsto \sigma_s(S_s)\in\mathscr P$ (or $J\ni s\mapsto \sigma_s(S^{\pm}_s)\in\mathscr P$ if $J\subset (\min\{\beta^t,\beta^b\},\min\{\beta^l,\beta^r\})$) has the form $s\mapsto \mathbf P(s)$ with $|W|(\mathscr X_{\mathbf P}\cup\mathscr Y_{\mathbf P}\cup\{\ell\})(s)> 0$ for every $s\in J$ and
for all $\mathbf x\in \mathscr X_{\mathbf P}$, $\mathbf y\in \mathscr Y_{\mathbf P}$ and $s\in J$ we have
\begin{align*}
&[\mathbf x,\ell](s)\leq 0\ \text{ and }\ [\mathbf y,\ell](s)>0\ \text{ if }\  J\subset(\beta^t,b),\\
&[\mathbf x,\ell](s)\geq 0\ \text{ and }\ [\mathbf y,\ell](s)<0\ \text{ if }\  J\subset(b,a).
\end{align*}
Applying Theorem~\ref{thm:mainsurf} this gives the required conclusion.
\end{proof}

\section*{Acknowledgments}
The author would like to thank Corinna Ulcigrai and Barak Weiss for fruitful discussions in the initial stage of the project.
They both had an invaluable impact on solving the problem. We acknowledge  the  \emph{Centre International de Rencontres Math\'ematiques} in Luminy for hospitality where preliminary discussions were conducted.

\end{document}